\documentclass[reqno]{amsart}

\usepackage[a4paper,width={15cm},left=2.5cm,bottom=3.5cm, top=3.5cm]{geometry}

\usepackage{graphicx,subfigure,color}

\usepackage{relsize,exscale}

\numberwithin{equation}{section}

\usepackage[latin1]{inputenc}
\usepackage[english]{babel}

\usepackage{amsmath,amsthm,amsfonts,latexsym,amssymb}
\usepackage[colorlinks]{hyperref}
\hypersetup{linkcolor=blue,citecolor=blue,filecolor=black,urlcolor=blue}

\usepackage{color}

{ \theoremstyle{plain}
\newtheorem{theorem}{Theorem}[section]
\newtheorem{proposition}[theorem]{Proposition}
\newtheorem{lemma}[theorem]{Lemma}
\newtheorem{corollary}[theorem]{Corollary}
  \theoremstyle{remark}
\newtheorem{remark}[theorem]{Remark}
  \theoremstyle{definition}
\newtheorem{definition}[theorem]{Definition}
}

\usepackage[normalem]{ulem}

\begin{document}
\subjclass[2010]{}


\title[]{Some evaluations of the fractional $p$-Laplace operator on radial functions }

\author[F. Colasuonno]{Francesca Colasuonno}
\email{francesca.colasuonno@unibo.it}

\author[F. Ferrari]{Fausto Ferrari}
\email{fausto.ferrari@unibo.it}

\address{
Dipartimento di Matematica\newline\indent
Alma Mater Studiorum Universit\`a di Bologna\newline\indent
piazza di Porta S. Donato, 5\newline\indent
40126 Bologna, Italy
}

\author[P. Gervasio]{Paola Gervasio}
\email{paola.gervasio@unibs.it}
\address{Dipartimento di Ingegneria Civile, Architettura, Territorio, Ambiente e di Matematica\newline\indent
Universit\`a degli Studi di Brescia\newline\indent
via Branze, 43\newline\indent
25123 Brescia, Italy
}

\author[A. Quarteroni]{Alfio Quarteroni}
\email{alfio.quarteroni@polimi.it}
\address{MOX, Dipartimento di Matematica,\newline\indent
Politecnico di Milano\newline\indent
via Bonardi, 9\newline\indent
20133 Milano, Italy\newline\indent
and\newline\indent
EPFL Lausanne, Switzerland (Professor Emeritus)
}
\date{\today}

\begin{abstract} 
We face a rigidity problem for the fractional $p$-Laplace operator to extend to this new framework some tools useful for the linear case. It is known that $(-\Delta)^s(1-|x|^{2})^s_+$ and $-\Delta_p(1-|x|^{\frac{p}{p-1}})$ are constant functions in $(-1,1)$ for fixed $p$ and $s$. We evaluated $(-\Delta_p)^s(1-|x|^{\frac{p}{p-1}})^s_+$ proving that it is not constant in $(-1,1)$ for some $p\in (1,+\infty)$ and $s\in (0,1)$.  This conclusion is obtained numerically thanks to the use of very accurate Gaussian numerical quadrature formulas.
\end{abstract}

\maketitle

\section{Introduction}
In this paper we wish to investigate, in a nonlocal nonlinear framework, some tools that have proved to be particularly useful for obtaining symmetry results for local operators.

It is well known that one of the crucial steps for applying the moving plane method to overdetermined problems \`a la Serrin is via a comparison principle. In the nonlocal setting there are, in the literature, several versions of comparison principles: in the linear case $p=2$, they follow by linearity from the maximum principle, while in the nonlinear case $p\neq 2$, they are more difficult to obtain. Strong maximum principles for fractional Laplacian-type operators have been proved in \cite{MN}, a weak maximum principle for antisymmetric solutions of problems governed by the fractional Laplacian can be found in \cite{FJ} (see also \cite{JW} for more general nonlocal operators), and a version of the strong maximum principle in the case of nonlocal Neumann boundary conditions can be found in \cite{CC}. For the fractional $p$-Laplacian operator, we refer to \cite{LL} for a weak comparison principle (see also \cite{IMS}), and to \cite{J} for a strong comparison principle; while some versions of the strong maximum principle and Hopf lemma can be found in \cite{CL,DPQ}. 
In the first part of this paper we revisit some results concerning the comparison principle for the fractional $p$-Laplace operator in bounded domains and prove a slightly new version of the strong comparison principle in Theorem \ref{thm:main}.

In the second part of the paper we address the study of the $p$-fractional torsion problem 
\begin{equation}\label{eq:rigidity}
\begin{cases}
(-\Delta_p)^s u=1 \quad&\mbox{in }B,\\
u=0&\mbox{in }\mathbb R^N\setminus B,
\end{cases}
\end{equation} 
where $s\in(0,1)$, $p>1$, $B\subset\mathbb R^N$ ($N\ge 1$) is a ball,
\begin{equation}\label{eq:classical-Deltaps0}
(-\Delta_p)^su(x):=c_{N,s,p}\lim_{\varepsilon\to 0^+}\int_{(B_\varepsilon(x))^c}\frac{|u(x)-u(y)|^{p-2}(u(x)-u(y))}{|x-y|^{N+ps}}dy,\quad x \in B
\end{equation}
denotes the fractional $p$-Laplace operator, and $c_{N,s,p}>0$ a suitable normalization constant. Such a problem admits a unique solution, which is radial and radially non-decreasing, cf. \cite[Lemma 4.1]{IMS}, but whose analytic expression is not known in the nonlocal nonlinear case: $s\in (0,1)$ and $p\neq 2$. 

On the other hand, in the local case $s=1$, it is easy to prove that the function $(1-|x|^m)$, with $m=\frac{p}{p-1}$, has constant $p$-Laplacian in $(-1,1)$, see for instance \cite{CF}. Moreover, in the linear case $p=2$, it has been proved that
$(1-x^2)_+^s$ satisfies $(-\Delta)^s(1-x^2)_+^s=\mathrm{Const.}$ in $(-1,1)$, see \cite{Dyda}. In view of these two results, and recalling that $(-\Delta_p)^su(x)\to -\Delta_pu(x),$ when $s\to 1^-,$ see \cite{IH}, as well as, of course, that $(-\Delta_2)^su(x)=(-\Delta)^su(x),$ it would be interesting to check whether the function $(1-|x|^m)_+^s$ may satisfy the equation
  $$(-\Delta_p)^s(1-|x|^m)_+^s=\mathrm{Const.}>0,$$
for every $x\in (-1,1)\subset \mathbb{R}$. In fact, the construction of the solution of the problem \eqref{eq:rigidity} would follow easily by a homogeneity argument. This result however does not hold true. As a matter of fact,  we prove that there exist $p>2,$ $s\in (0,1),$ $x_1,x_2\in (-1,1)$ such that $x_1\not=x_2$ and  $(-\Delta_p)^s(1-|x_1|^m)_+^s\not=(-\Delta_p)^s(1-|x_2|^m)_+^s.$
Our proof follows by investigating the value of
 $$((-\Delta_p)^s(1-|x|^m)_+^s)_{|x=0}=2c_{1,s,p}\left(\frac{1}{sp}+\int_0^1\frac{(1-(1-y^m)^s)^{p-1}}{y^{1+sp}}dy\right),$$
 where the exact value of $c_{1,s,p}$ is given in Section \ref{Chap3}.

The paper is organized as follows. In Section \ref{Chap2} we deduce a strong comparison principle that holds for the fractional $p$-Laplace operator in any dimension $N\ge 1$. Notice that, in the local case, a similar result has been proved only in dimension $N=2,$ see \cite{M}. In Section \ref{Chap3} we prove, by following a different strategy with respect to \cite{Dyda}, that the $s$-fractional Laplace operator of $(1-x^2)^s_+$ in $(-1,1)$ is constant. In  Section \ref{Chap4} we prepare the ground for a numerical evaluation of the $s$-fractional $p$-Laplacian of $(1-x^m)^s_+$, proving integrability properties, see Propositions \ref{prop:Wsp} and \ref{prop:gxy}, that are useful to yield error estimates for our numerical integration formaulae.
Finally, in Section \ref{Chap5} we show, by computing numerically the integral in (\ref{eq:classical-Deltaps0}), that there exist $p\neq 2$ and $s\in(0,1)$ such that the $s$-fractional $p$-Laplace operator of $(1-x^m)^s_+$ is not constant in $(-1,1)$.

\section{The strong comparison principle for the fractional $p$-Laplacian}\label{Chap2}
In this section, we consider the following system of inequalities
\begin{equation}\label{eq:main}
\begin{cases}
(-\Delta_p)^s u + q(x)|u|^{p-2}u \le (-\Delta_p)^s v + q(x)|v|^{p-2}v \quad&\mbox{in }\Omega,\\
u \le v&\mbox{in }\mathbb R^N,
\end{cases}
\end{equation}
where $s\in(0,1)$, $p>1$, $\Omega\subset\mathbb R^N$ ($N\ge 1$) is a bounded domain, $q\in L^\infty(\Omega)$, and $(-\Delta_p)^s$ denotes the fractional $p$-Laplacian, which, on smooth functions $u$, can be written as
\begin{equation}\label{eq:classical-Deltaps}
(-\Delta_p)^su(x):=c_{N,s,p}\lim_{\varepsilon\to 0^+}\int_{(B_\varepsilon(x))^c}\frac{|u(x)-u(y)|^{p-2}(u(x)-u(y))}{|x-y|^{N+ps}}dy,\quad x\in\Omega,
\end{equation}
where $c_{N,s,p}>0$ is the usual normalization constant introduced in the Introduction.

We will prove the following strong comparison principle.
\begin{theorem}\label{thm:main}
Let $(u,v)$ be a weak solution of \eqref{eq:main}. If $u,\,v\in C(\Omega)$ and
\begin{equation}\label{eq:integrability}
\int_\Omega\int_\Omega\frac{\Big||v(x)-v(y)|^{p-2}(v(x)-v(y))-|u(x)-u(y)|^{p-2}(u(x)-u(y))\Big|}{|x-y|^{N+sp}}dxdy<\infty,
\end{equation} 
then either $u<v$ in $\Omega$ or $u\equiv v$ in $\mathbb R^N$.
\end{theorem}
The proof of this theorem is based on an argument first introduced in \cite{M}. We observe that the previous strong comparison principle for $(-\Delta_p)^s$, \cite[Theorem 1.1]{J}, requires different regularity assumptions on $u$ and $v$ and uses a different proof technique.
Before proving Theorem \ref{thm:main}, we introduce the functional spaces and the main definitions that will be useful to work with weak solutions, and prove a preliminary lemma.

For every $s\in (0,1)$ and $p\in (1,\infty)$, we define
$$
W^{s,p}(\mathbb{R}^N):=\left\{u\in L^p(\mathbb R^N)\,:\,\int_{\mathbb R^N}\int_{\mathbb R^N}\frac{|u(x)-u(y)|^p}{|x-y|^{N+sp}}dxdy<\infty\right\},
$$
$$
W^{s,p}_0(\Omega):=\{u\in W^{s,p}(\mathbb{R}^N)\,:\,u\equiv 0 \mbox{ in }\mathbb R^N\setminus\Omega\},
$$
$$
\widetilde{W}^{s,p}(\Omega):=\left\{u\in L^p_{\mathrm{loc}}(\mathbb R^N)\,:\,\int_{\Omega}\int_{\mathbb R^N}\frac{|u(x)-u(y)|^p}{|x-y|^{N+sp}}dxdy<\infty\right\}.
$$

\begin{definition}
A function $u\in \widetilde{W}^{s,p}(\Omega)$ is a weak solution of $(-\Delta_p)^s u \ge (\le) 0$ in $\Omega$ if 
$$
\int_{\mathbb R^N}\int_{\mathbb R^N}\frac{|u(x)-u(y)|^{p-2}(u(x)-u(y))(\varphi(x)-\varphi(y))}{|x-y|^{N+sp}}dxdy\ge(\le)0
$$
holds for every $0 \le \varphi\in W^{s,p}_0(\Omega)$.
Consequently, a function $u\in \widetilde{W}^{s,p}(\Omega)$ is a weak solution of $(-\Delta_p)^s u = 0$ in $\Omega$ if 
$$
\int_{\mathbb R^N}\int_{\mathbb R^N}\frac{|u(x)-u(y)|^{p-2}(u(x)-u(y))(\varphi(x)-\varphi(y))}{|x-y|^{N+sp}}dxdy = 0
$$
holds for every $\varphi\in W^{s,p}_0(\Omega)$.

A couple $(u,v)\in (\widetilde{W}^{s,p}(\Omega))^2$ is a weak solution of \eqref{eq:main} if the inequality
$$
\begin{aligned}
c_{N,s,p}&\int_{\mathbb R^N}\int_{\mathbb R^N}\frac{|u(x)-u(y)|^{p-2}(u(x)-u(y))(\varphi(x)-\varphi(y))}{|x-y|^{N+sp}}dxdy\\
&\hspace{7cm} + \int_\Omega q(x)|u(x)|^{p-2}u(x)\varphi(x) dx\\
& \le c_{N,s,p}\int_{\mathbb R^N}\int_{\mathbb R^N}\frac{|v(x)-v(y)|^{p-2}(v(x)-v(y))(\varphi(x)-\varphi(y))}{|x-y|^{N+sp}}dxdy\\
&\hspace{7cm} + \int_\Omega q(x)|v(x)|^{p-2}v(x)\varphi(x) dx\\
\end{aligned}
$$
holds for every $0\le\varphi\in W^{s,p}_0(\Omega)$, and $u\le v$ a.e. in $\mathbb R^N$.
\end{definition}

\begin{lemma}\label{lem:FLCV-positive}
If $f\in L^1_{\mathrm{loc}}(\Omega)$ is such that 
$$
\int_\Omega f(x)\varphi(x)dx\ge 0 \quad\mbox{for every }0\le \varphi \in C^\infty_c(\Omega),
$$
then $f\ge 0$ a.e. in $\Omega$.
\end{lemma}
\begin{proof}
Let $\Omega=\bigcup_{n=1}^\infty\Omega_n$, with $\Omega_n\subset\Omega_{n+1}$, 
and let $K_n$ be compact sets such that 
\begin{equation}\label{eq:set-incl}
K_n\subset\Omega_n\subset K_{n+1}\quad\mbox{for every }n\in\mathbb N.
\end{equation}
Since $f\in L^1_{\mathrm{loc}}(\Omega)$, $f\in L^1(K_n)$ and consequently, by \eqref{eq:set-incl}, $f\in L^1(\Omega_n)$ for every $n$. In particular, $f^+,\,f^-\in L^1(\Omega_n)$ and also $1_{\{f^->0\}}\in L^1(\Omega_n)$. Now, fix $n\in\mathbb N$. By density, there exists a sequence $(\varphi_j)\subset C^\infty_c(\Omega_n)$ such that $\varphi_j\to 1_{\{f^->0\}}$ in $L^1(\Omega_n)$. Therefore, passing if necessary to a subsequence, and using the Dominated Convergence Theorem, we get
\begin{equation}\label{eq:fphij}
\lim_{j\to\infty}\int_{\Omega_n}f\varphi_jdx=\int_{\Omega_n}\lim_{j\to\infty}(f^+-f^-)\varphi_jdx=-\int_{\Omega_n\cap \{f^->0\}}f^- dx.
\end{equation}
Now, by assumption, for every $j\in\mathbb N$, $\int_{\Omega_n}f\varphi_j dx= \int_{\Omega}f\varphi_j dx\ge 0$,
being $\varphi_j\in C^\infty_c(\Omega_n)\subset C^\infty_c(\Omega)$.
Hence, by \eqref{eq:fphij}, 
\begin{equation}\label{eq:int-f-}
\int_{\Omega_n\cap \{f^->0\}}f^- dx\le 0.
\end{equation}
We can now pass to the limit in $n$ to obtain
\begin{align*}
0&\ge \lim_{n\to\infty}\int_{\Omega_n\cap\{f^->0\}}f^- dx=\lim_{n\to\infty}\int_{\{f^->0\}}f^- 1_{\Omega_n} dx=
\int_{\{f^->0\}}f^- 1_{\Omega} dx\\
&=\int_{\{f^->0\}}f^- dx\ge 0,
\end{align*}
where we have used \eqref{eq:int-f-}, $1_{\Omega_n}\to 1_{\Omega}$ a.e., and the Monotone Convergence Theorem.
This immediately gives $|\{f^->0\}|=0$ and concludes the proof. 
\end{proof}

\begin{proof}[$\bullet$ Proof of Theorem \ref{thm:main}.] 
We introduce the notation $A(h):=|h|^{p-2}h$ for every $h\in \mathbb R$. Since $(u,v)$ is a weak solution of \eqref{eq:main}, we get for every $0\le\varphi\in C^\infty_c(\Omega)$
\begin{equation}\label{eq:subtracting}
\begin{aligned}
I:=\frac{1}{c_{N,s,p}}&\int_\Omega q(x)(|u(x)|^{p-2}u(x)-|v(x)|^{p-2}v(x)|)\varphi(x)dx\\
&\le\int_{\mathbb{R}^N}\int_{\mathbb{R}^N}\{|v(x)-v(y)|^{p-2}(v(x)-v(y))-|u(x)-u(y)|^{p-2}(u(x)-u(y))\}\\
& \hspace{8.5cm} \cdot\frac{\varphi(x)-\varphi(y)}{|x-y|^{N+sp}}dx dy\\
&= \int_{\mathbb{R}^N}\int_{\mathbb{R}^N}\{A(v(x)-v(y))-A(u(x)-u(y))\}\frac{\varphi(x)-\varphi(y)}{|x-y|^{N+sp}}dx dy.
\end{aligned}
\end{equation}
On the other hand, let $u_t:=tv+(1-t)u$ for every $t\in [0,1]$ and $w:=v-u$, then by straightforward calculations we have
\begin{align*}
A(v(x)-v(y))-A(u(x)-u(y))&=\int_0^1\frac{d}{dt}A(u_t(x)-u_t(y))dt\\
&=(p-1)\left(\int_0^1|u_t(x)-u_t(y)|^{p-2}dt\right)(w(x)-w(y))\\
&=:a(x,y)(w(x)-w(y)).
\end{align*}
We observe that $a(x,y)=a(y,x)$ for every $x,\,y\in\mathbb R^N$.
Hence, continuing the estimate in \eqref{eq:subtracting} and using that $\varphi\equiv 0$ in $\mathbb R^N\setminus\Omega$, we have
\begin{align*}
I&\le \int_{\mathbb{R}^N}\int_{\mathbb{R}^N}a(x,y)(w(x)-w(y))\frac{(\varphi(x)-\varphi(y))}{|x-y|^{N+sp}}dx dy\\
&= \int_{\mathbb R^N\setminus\Omega}\int_{\Omega}\dots dx dy+  \int_{\Omega}\int_{\Omega}\dots dx dy
\int_{\Omega}\int_{\mathbb R^N\setminus\Omega}\dots dx dy.
\end{align*}
By the symmetry of $a(\cdot,\cdot)$, we notice that the first and the third integral in the last expression are equal, so that we can write
\begin{equation}\label{eq:manipulation}
\begin{aligned}
I&\le 2\int_{\mathbb{R}^N\setminus\Omega}\left(\int_{\Omega}a(x,y)(w(x)-w(y))\frac{\varphi(x)}{|x-y|^{N+sp}}dx\right) dy\\
&\qquad \qquad +\int_{\Omega}\int_{\Omega} a(x,y)(w(x)-w(y))\frac{(\varphi(x)-\varphi(y))}{|x-y|^{N+sp}}dx dy\\
&=\int_{\mathbb{R}^N}\left(\int_{\Omega}a(x,y)(w(x)-w(y))\frac{\varphi(x)}{|x-y|^{N+sp}}dx\right) dy\\
&\qquad \qquad 
+\int_{\mathbb{R}^N\setminus\Omega}\left(\int_{\Omega}a(x,y)(w(x)-w(y))\frac{\varphi(x)}{|x-y|^{N+sp}}dx\right) dy\\
&\qquad \qquad \qquad
-\int_{\Omega}\left(\int_{\Omega}a(x,y)(w(x)-w(y))\frac{\varphi(y)}{|x-y|^{N+sp}}dx\right) dy.
\end{aligned}
\end{equation}
As for the last integral, in view of \eqref{eq:integrability}, we can manipulate it in the following way
\begin{align*}
-\int_{\Omega}&\left(\int_{\Omega}a(x,y)(w(x)-w(y))\frac{\varphi(y)}{|x-y|^{N+sp}}dx\right) dy
\\
&=
-\int_{\Omega}\left(\int_{\Omega}a(x,y)(w(x)-w(y))\frac{\varphi(y)}{|x-y|^{N+sp}}dy\right) dx\\
&=\int_{\Omega}\left(\int_{\Omega}a(x,y)(w(x)-w(y))\frac{\varphi(x)}{|x-y|^{N+sp}}dx\right) dy,
\end{align*}
therefore, we can sum up the last two integrals in \eqref{eq:manipulation}
to get in conclusion
\begin{equation}\label{eq:RHS}
\begin{aligned}
I&\le2\int_{\mathbb{R}^N}\left(\int_{\Omega}a(x,y)\frac{(w(x)-w(y))}{|x-y|^{N+sp}}\varphi(x)dx\right) dy\\		
&= 2\int_{\Omega}\left(\int_{\mathbb{R}^N}a(x,y)\frac{(w(x)-w(y))}{|x-y|^{N+sp}}dy\right) \varphi(x) dx
\end{aligned}
\end{equation}
for every $0\le \varphi \in C^\infty_c(\Omega)$. Arguing in a similar way, we can re-write the integral $I$ as follows
\begin{equation}\label{eq:LHS}
\begin{aligned}
I&=\frac{1}{c_{N,s,p}}\int_\Omega q(x)(A(u(x))-A(v(x)))\varphi(x)dx\\
&=\frac{1}{c_{N,s,p}}\int_\Omega q(x)\left(-\int_0^1\frac{d}{dt}A(u_t(x))dt\right)\varphi(x)dx\\
&=-\int_\Omega q(x)\left(\frac{(p-1)\int_0^1|u_t(x)|^{p-2}dt}{c_{N,s,p}}\right)w(x)\varphi(x)dx=: -\int_\Omega q(x) b(x) w(x)\varphi(x)dx.
\end{aligned}
\end{equation}
Combining together \eqref{eq:RHS} and \eqref{eq:LHS}, we get
$$
\int_{\Omega}\left(2\int_{\mathbb{R}^N}a(x,y)\frac{(w(x)-w(y))}{|x-y|^{N+sp}}dy+q(x) b(x) w(x)\right) \varphi(x) dx\ge 0
$$
for every $0\le \varphi \in C^\infty_c(\Omega)$. Thus, by Lemma \ref{lem:FLCV-positive}, this implies that 
$$
2\int_{\mathbb R^N}a(x,y)\frac{w(x)-w(y)}{|x-y|^{N+sp}}dy\ge -q(x) b(x) w(x) \quad\mbox{for a.e. }x\in \Omega.
$$
Now, suppose by contradiction that there exists $x_0\in \Omega$ such that $w(x_0)=0$, then 
\begin{equation}
\label{eq:eq0}
\int_{\mathbb R^N}a(x_0,y)\frac{-w(y)}{|x_0-y|^{N+sp}}dy \ge 0.
\end{equation}
Since $w=v-u\ge 0$ a.e. in $\mathbb R^N$ and $a(x,y)\ge 0$ for a.e. $x,\,y\in \mathbb R^N$, 
\eqref{eq:eq0} implies that $a(x_0,y)w(y)=0$ for a.e. $y\in\mathbb R^N$. 

We are now ready to conclude. We observe that, if $a(x_0,y)=0$ for some $y\in \mathbb R^N$, then $w(y)=0$. Indeed, by straightforward calculations, if
$$
a(x_0,y)=(p-1)\int_0^1|u(x_0)-u(y)+t(u(y)-v(y))|^{p-2}dt=0,
$$
then
$$
u(x_0)-u(y)+t(u(y)-v(y))=0\quad\mbox{for every }t\in [0,1],
$$
which gives $u(y)=v(y)$, or equivalently $w(y)=0$. So, we have proved that $a(x_0,y)w(y)=0$ for a.e. $y\in\mathbb R^N$ is equivalent to  $w(y)=0$ for a.e. $y\in\mathbb R^N$, which concludes the proof.
\end{proof}

Arguing as in \cite[Lemma 9]{LL}, we have the following weak comparison principle. We stress that, with respect to Theorem \ref{thm:main}, we need to ask $u$ to be continuous in the whole space and $q$ to be non-negative.
\begin{lemma}
Let $(u,v)$ be a weak solution of 
$$
\begin{cases}
(-\Delta_p)^s u + q(x)|u|^{p-2}u \le (-\Delta_p)^s v + q(x)|v|^{p-2}v\quad&\mbox{in }\Omega,\\
u \le v&\mbox{in }\mathbb R^N\setminus\Omega,
\end{cases}
$$
where $0\le q\in L^\infty(\Omega)$.
If $u,\,v\in C(\mathbb R^N)$, then $u\le v$ also in $\Omega$.
\end{lemma}
\begin{proof}
Reasoning as in the first part of the proof of Theorem \ref{thm:main}, we get for every $0\le \varphi\in W^{s,p}_0(\Omega)$
\begin{equation}\label{eq:ineq}
\int_{\mathbb R^N}\int_{\mathbb R^N} a(x,y)(w(x)-w(y))\frac{\varphi(x)-\varphi(y)}{|x-y|^{N+sp}}dx\,dy\ge-\int_\Omega q(x)b(x)w(x)\varphi(x)dx,
\end{equation}
with the same definitions for $a$, $b$, and $w=v-u$. Now, following the idea in \cite[Lemma 9]{LL}, we choose $\varphi:=(u-v)^+=w^-$ and observe that 
$$
w\varphi=(w^+-w^-)w^-=-(w^-)^2\le 0.
$$
Hence, putting $\varphi=w^-$, and using that $q\ge0$, we get from \eqref{eq:ineq}
$$\int_{\mathbb R^N}\int_{\mathbb R^N} a(x,y)(w(x)-w(y))\frac{w^-(x)-w^-(y)}{|x-y|^{N+sp}}dx\,dy\ge0.
$$
The proof now can be completed exactly as in \cite[Lemma 9]{LL}.
\end{proof}

Combining the previous lemma with Theorem \ref{thm:main}, we get the following. 
\begin{corollary}
Let $(u,v)$ be a weak solution of 
$$
\begin{cases}
(-\Delta_p)^s u + q(x)|u|^{p-2}u \le (-\Delta_p)^s v + q(x)|v|^{p-2}v\quad&\mbox{in }\Omega,\\
u \le v&\mbox{in }\mathbb R^N\setminus\Omega,
\end{cases}
$$
where $0\le q\in L^\infty(\Omega)$. If $u,\,v\in C(\mathbb R^N)$ and \eqref{eq:integrability} holds, then 
either $u<v$ in $\Omega$ or $u\equiv v$ in $\mathbb R^N$.
\end{corollary}

\section{The fractional torsion problem}\label{Chap3}
Let $s\in(0,1)$. In this section we consider the following problem
\begin{equation}\label{eq:torsion-pb}
\begin{cases}
(-\Delta)^s u=1\quad&\mbox{in }B_R\\
u=0&\mbox{in }\mathbb R^N\setminus B_R,
\end{cases}
\end{equation}
where $B_R\subset \mathbb R^N$ is the ball of radius $R$ centered at the origin.

\begin{remark}
We observe that, at least formally, the $N$-dimensional fractional Laplacian $(-\Delta)^s_{N}$ of a function $u:\mathbb R^N\to \mathbb R$ can be expressed in terms of the 1-dimensional fractional Laplacian $(-\Delta)^s_1$ of related functions of one variable. Indeed, denoting simply $c_{N,s}:=c_{N,s,2}$, we get for every $x\in\mathbb R^N$
$$\begin{aligned}
\frac{1}{c_{N,s}}(-\Delta)^s_N u(x)&=\lim_{\varepsilon\to 0^+}\int_{\mathbb R^N\setminus B_\varepsilon(0)}\frac{u(x)-u(y)}{|x-y|^{N+2s}}dy=\lim_{\varepsilon\to 0^+}\int_\varepsilon^{+\infty}\left(\int_{\partial B_t(x)}\frac{u(x)-u(y)}{|x-y|^{N+2s}}d\sigma(y)\right)dt\\
&=\lim_{\varepsilon\to 0^+}\int_\varepsilon^{+\infty}\left(\int_{\partial B_1(x)}\frac{u(x)-u(x-t\nu)}{t^{2s+1}}d\sigma(\nu)\right)dt\\
&=\lim_{\varepsilon\to 0^+}\left\{\int_{\partial B_1(x)\cap\{x_N>0\}}\left(\int_\varepsilon^{+\infty}\frac{u(x)-u(x-t\nu)}{t^{2s+1}}dt\right)d\sigma(\nu)\right.\\
&\phantom{=}\qquad\quad\left.+\int_{\partial B_1(x)\cap\{x_N<0\}}\left(\int_\varepsilon^{+\infty}\frac{u(x)-u(x-t\nu)}{t^{2s+1}}dt\right)d\sigma(\nu)\right\}\\
&=\lim_{\varepsilon\to 0^+}\int_{\partial B_1(x)\cap\{x_N>0\}}\left(\int_{\mathbb R\setminus(-\varepsilon,\varepsilon)}\frac{u(x)-u(x-t\nu)}{|t|^{2s+1}}dt\right)d\sigma(\nu)\\
&=\int_{\partial B_1(x)\cap\{x_N>0\}}\left(\lim_{\varepsilon\to 0^+}\int_{\mathbb R\setminus(-\varepsilon,\varepsilon)}\frac{u(x)-u(x-t\nu)}{|t|^{2s+1}}dt\right)d\sigma(\nu)\\
&=\int_{\partial B_1(x)\cap\{x_N>0\}}\frac{1}{c_{1,s}}(-\Delta)^s_1 \psi_{\nu,x}(0)d\sigma(\nu),
\end{aligned}
$$
where $\psi_{\nu,x}(t):=u(x-t\nu)$ for every $t\in\mathbb R$. In particular, for $u_N(x)=(1-|x|^2)^s_+$, $\psi_{\nu,x}(t)=u_1(|x-t\nu|)$, and so, once it is proved that $(-\Delta)^s_1u_1$ is constant, one has immediately that also $(-\Delta)^s_Nu_N$ is constant. 
\end{remark}

In the light of the previous remark, from now on in the paper we consider only the case of dimension $N=1$, and drop all subscripts referring to the dimension. Moreover, for the sake of simplicity, we take the radius $R$ to be 1. In this setting, we give an alternative proof of the fact that the solution of \eqref{eq:torsion-pb} is given by $v_s(x):=\frac{\sin(\pi s)}{\pi c_s}(1-x^2)^s_+$, where  $c_s$ is the normalization constant for the fractional Laplacian in dimension one and is given by $c_s:=\frac{2^{2s}}{\sqrt{\pi}}\frac{\Gamma\left(\frac{1+2s}{2}\right)}{\Gamma(1-s)}s$, cf. for instance \cite[Remark 3.11]{CS}. We refer to \cite{Dyda} for a previous proof. 

\begin{theorem}\label{thm:frac-tor}
Let $N=1$ and $v_s(x):=\frac{\sin(\pi s)}{\pi c_s}u_s(x)$, with $u_s(x):=(1-x^2)^s_+$. Then $v_s$ is a $C^s([-1,1])$ solution of \eqref{eq:torsion-pb}.
\end{theorem}
\begin{proof} For every $x\in\mathbb R\setminus(-1,1)$, $u_s(x)=0$. Moreover, for every $x\in (-1,1)$,
$$
\begin{aligned}
(-\Delta)^su_s(-x)&=c_{s}\lim_{\varepsilon\to 0^+}\int_{\mathbb R\setminus(-x-\varepsilon,-x+\varepsilon)}\frac{u_s(-x)-u_s(y)}{|-x-y|^{1+2s}}dy\\
&=c_{s}\lim_{\varepsilon\to 0^+}\int_{\mathbb R\setminus(-x-\varepsilon,-x+\varepsilon)}\frac{u_s(x)-u_s(-y)}{|x+y|^{1+2s}}dy\\
&=c_{s}\lim_{\varepsilon\to 0^+}\int_{\mathbb R\setminus(x-\varepsilon,x+\varepsilon)}\frac{u_s(x)-u_s(z)}{|x-z|^{1+2s}}dz=(-\Delta)^su_s(x).
\end{aligned}
$$
Now, let $x\in(0,1)$ and $\varepsilon>0$, then
$$
\begin{aligned}
\int_{\mathbb R\setminus (x-\varepsilon,x+\varepsilon)}\frac{u_s(x)-u_s(y)}{|x-y|^{1+2s}}dy
&=\int_{-1}^{x-\varepsilon}\frac{(1-x^2)^s-(1-y^2)^s}{|x-y|^{1+2s}}dy\\
&\quad+\int_{x+\varepsilon}^{1}\frac{(1-x^2)^s-(1-y^2)^s}{|x-y|^{1+2s}}dy\\
&\quad+(1-x^2)^s\int_{\mathbb R\setminus (-1,1)}\frac{1}{|x-y|^{1+2s}}dy=:I_1(x)+I_2(x)+I_3(x).
\end{aligned}
$$
As for the last integral, we immediately get
$$
\begin{aligned}
\frac{I_3(x)}{(1-x^2)^s}&=\frac{1}{x^{1+2s}}\left(\int_{-\infty}^{-1}\frac{1}{|1-y/x|^{1+2s}}dy+\int_1^{\infty}\frac{1}{|1-y/x|^{1+2s}}dy\right)\\
&=\frac{1}{2sx^{2s}}\left[\left(1+\frac{1}{x}\right)^{-2s}+\left(\frac{1}{x}-1\right)^{-2s}\right]=\frac{1}{2s}\left[\frac{1}{(1-x)^{2s}}+\frac{1}{(1+x)^{2s}}\right].
\end{aligned}
$$
We manipulate and integrate by parts $I_1(x)$ to obtain
\begin{equation}\label{eq:I1}
\begin{aligned}
I_1(x)&=(1-x^2)^s\int_{-1}^{x-\varepsilon}\frac{1}{(x-y)^{1+2s}}dy-\int_{-1}^{x-\varepsilon}\frac{(1-y^2)^s}{(x-y)^{1+2s}}dy\\
&=\frac{(1-x^2)^s}{2s}\left(\frac{1}{\varepsilon^{2s}}-\frac{1}{(1+x)^{2s}}\right)-\left\{
\frac{(1-(x-\varepsilon)^2)^s}{2s\varepsilon^{2s}}+\int_{-1}^{x-\varepsilon}\frac{(1-y^2)^{s-1}}{(x-y)^{2s}}y\,dy\right\}.
\end{aligned}
\end{equation}
Similarly, for $I_2(x)$ we have
\begin{equation}\label{eq:I2}
\begin{aligned}
I_2(x)&=(1-x^2)^s\int_{x+\varepsilon}^1\frac{1}{(y-x)^{1+2s}}dy-\int_{x+\varepsilon}^1\frac{(1-y^2)^s}{(y-x)^{1+2s}}dy\\
&=\frac{(1-x^2)^s}{2s}\left(\frac{1}{\varepsilon^{2s}}-\frac{1}{(1-x)^{2s}}\right)-\left\{
\frac{(1-(x+\varepsilon)^2)^s}{2s\varepsilon^{2s}}-\int_{x+\varepsilon}^1\frac{(1-y^2)^{s-1}}{(y-x)^{2s}}y\,dy\right\}.
\end{aligned}
\end{equation}
Now, it is straightforward to see that, as $\varepsilon\to0^+$, 
$$
\begin{gathered}
(1-(x-\varepsilon)^2)^s=(1-x^2)^s+\frac{2s x}{(1-x^2)^{1-s}}\varepsilon+O(\varepsilon^2),\\
(1-(x+\varepsilon)^2)^s=(1-x^2)^s-\frac{2s x}{(1-x^2)^{1-s}}\varepsilon+O(\varepsilon^2).
\end{gathered}
$$
Thus, combining them with \eqref{eq:I1} and \eqref{eq:I2}, we obtain as $\varepsilon\to 0^+$
$$
I_1(x)=-\frac{(1-x^2)^s}{2s(1+x)^{2s}}-
\frac{\varepsilon^{1-2s}x}{(1-x^2)^{1-s}}-\int_{-1}^{x-\varepsilon}\frac{(1-y^2)^{s-1}}{(x-y)^{2s}}y\,dy+O(\varepsilon^{2(1-s)})
$$
and
$$
I_2(x)=-\frac{(1-x^2)^s}{2s(1-x)^{2s}}+\frac{\varepsilon^{1-2s}x}{(1-x^2)^{1-s}}+\int_{x+\varepsilon}^1\frac{(1-y^2)^{s-1}}{(y-x)^{2s}}y\,dy+O(\varepsilon^{2(1-s)}). 
$$
Altogether, we have
\begin{equation}\label{eq:deltasu}
\begin{aligned}
(-\Delta)^su_s(x)&=c_s\lim_{\varepsilon\to 0^+}(I_1(x)+I_2(x)+I_3(x))\\
&=c_s\lim_{\varepsilon\to 0^+}\left(-\int_{-1}^{x-\varepsilon}\frac{(1-y^2)^{s-1}}{(x-y)^{2s}}y\,dy+\int_{x+\varepsilon}^1\frac{(1-y^2)^{s-1}}{(y-x)^{2s}}y\,dy+O(\varepsilon^{2(1-s)})\right).
\end{aligned}
\end{equation}

Now, we distinguish two cases depending on whether $s\in(0,\frac{1}{2})$ or $s\in[\frac{1}{2},1)$. 

$\bullet$ \underline{\it Case $s\in(0,\frac{1}{2})$}. In this case, all integrals involved in the fractional Laplacian of $u_s$ are convergent. So, in this case, we have 
$$
(-\Delta)^su_s(x)=c_s\left(-\int_{-1}^{x}\frac{(1-y^2)^{s-1}}{(x-y)^{2s}}y\,dy+\int_{x}^1\frac{(1-y^2)^{s-1}}{(y-x)^{2s}}y\,dy\right).
$$
Now, by the following change of variable $t=\frac{x-y}{1-xy}$, we have
$$
\begin{aligned}
\int_{-1}^{x}\frac{(1-y^2)^{s-1}}{(x-y)^{2s}}y\,dy&=-\frac{1}{(1-x^2)^{2s-1}}\int_0^1\frac{(x-t)((1-tx)^2-(x-t)^2)^{s-1}}{t^{2s}(1-tx)}dt\\
&=\frac{1}{(1-x^2)^{2s-1}}\int_0^1\frac{(t-x)[(1-t^2)(1-x^2)]^{s-1}}{t^{2s}(1-tx)}dt\\
&=\frac{1}{(1-x^2)^{s}}\int_0^1\frac{(t-x)(1-t^2)^{s-1}}{t^{2s}(1-tx)}dt
\end{aligned}
$$
and similarly
$$
\begin{aligned}
\int_{x}^1\frac{(1-y^2)^{s-1}}{(y-x)^{2s}}y\,dy&=\frac{1}{(1-x^2)^{s}}\int_{-1}^0\frac{(x-t)(1-t^2)^{s-1}}{t^{2s}(1-tx)}dt\\
&=\frac{1}{(1-x^2)^{s}}\int_0^1\frac{(x+t)(1-t^2)^{s-1}}{t^{2s}(1+tx)}dt.
\end{aligned}
$$
So that, summing up, we have
$$
\begin{aligned}
(-\Delta)^su_s(x)&=\frac{c_s}{(1-x^2)^{s}}\int_0^1\frac{(1-t^2)^{s-1}}{t^{2s}}\left[\frac{t-x}{1-tx}+\frac{t+x}{1+tx}\right]dt\\
&=2c_s(1-x^2)^{1-s}\int_0^1\frac{(1-t^2)^{s-1}}{t^{2s-1}(1-t^2x^2)}dt.
\end{aligned}
$$
We can now integrate using power series to get
$$
\begin{aligned}
(-\Delta)^su_s(x)&=2c_s(1-x^2)^{1-s}\int_0^1\frac{(1-t^2)^{s-1}}{t^{2s-1}}\sum_{k=0}^\infty(tx)^{2k} dt\\
&=2c_s(1-x^2)^{1-s}\sum_{k=0}^\infty \left(x^{2k}\int_0^1\frac{(1-t^2)^{s-1}}{t^{2s-1-2k}} dt\right)\\
&=2c_s(1-x^2)^{1-s}\sum_{k=0}^\infty \left(x^{2k}\frac{\Gamma(s)\Gamma(k-s+1)}{2\Gamma(k+1)}\right)\\
&=c_s\Gamma(s)\Gamma(1-s)(1-x^2)^{1-s}\sum_{k=0}^\infty (-1)^{k}{{s-1}\choose{k}}x^{2k}\\
&=c_s\Gamma(s)\Gamma(1-s)(1-x^2)^{1-s}(1-x^2)^{s-1}=c_s\Gamma(s)\Gamma(1-s)=c_s\frac{\pi}{\sin s\pi}
\end{aligned}
$$
where we have calculated the integral $\int_0^1\frac{(1-t^2)^{s-1}}{t^{2s-1-2k}} dt$ symbolically by Wolfram Mathematica \cite{Mathematica}, and we have used that $\Gamma(k+1)=k!$, the following property of the Gamma function (cf. for instance \cite[formula (1.47)]{SKM}), with $z=1-s$:
$$
\frac{\Gamma(z+k)}{\Gamma(z)}=(z)_k\quad\mbox{for }z>-k,\quad z\neq 0,-1,-2,\dots,
$$
the relation between definition of the Pochhammer symbol and the binomial coefficient (cf. for instance \cite[formula (1.48)]{SKM})
$$
{{-z}\choose{k}}=\frac{(-1)^k(z)_k}{k!},
$$
and finally that $\Gamma(s)\Gamma(1-s)=\frac{\pi}{\sin s\pi}$. The conclusion, in this case, follows at once for $v_s$, using the linearity of the fractional Laplacian.

$\bullet$ \underline{\it Case $s\in[\frac{1}{2},1)$}. In this case, the situation is technically more involved because, singularly taken, the integrals that appear in \eqref{eq:deltasu} are not convergent and one has to take carefully into account the cancellations. Using the change of variables $t:=\frac{x-y}{1-xy}$, we get
$$
-\int_{-1}^{x-\varepsilon}\frac{(1-y^2)^{s-1}}{(x-y)^{2s}}y\,dy=\frac{1}{(1-x^2)^s}\left(\int_{\frac{\varepsilon}{1-x(x-\varepsilon)}}^1\frac{(1-t^2)^{s-1}}{t^{2s-1}(1-tx)}dt-x\int_{\frac{\varepsilon}{1-x(x-\varepsilon)}}^1\frac{(1-t^2)^{s-1}}{t^{2s}(1-tx)}dt\right),
$$
where the first integral on the right-hand side is convergent as $\varepsilon\to 0^+$. Similarly, 
$$
\int_{x+\varepsilon}^1\frac{(1-y^2)^{s-1}}{(y-x)^{2s}}y\,dy=\frac{1}{(1-x^2)^s}\left(\int_{\frac{\varepsilon}{1-x(x+\varepsilon)}}^1\frac{(1-t^2)^{s-1}}{t^{2s-1}(1+tx)}dt+x\int_{\frac{\varepsilon}{1-x(x+\varepsilon)}}^1\frac{(1-t^2)^{s-1}}{t^{2s}(1+tx)}dt\right),
$$
where, again, the first integral on the right-hand side is convergent as $\varepsilon\to 0^+$.
Therefore, \eqref{eq:deltasu} can be re-written in the form
$$
\begin{aligned}
(-\Delta)^su_s(x)=&\frac{c_s}{(1-x^2)^s}\int_0^1 \frac{(1-t^2)^{s-1}}{t^{2s-1}}\left(\frac{1}{1-tx}+\frac{1}{1+tx}\right)dt\\
&+\frac{c_sx}{(1-x^2)^s}\lim_{\varepsilon\to 0^+}\left\{\int_{\frac{\varepsilon}{1-x(x+\varepsilon)}}^1\frac{(1-t^2)^{s-1}}{t^{2s}(1+tx)}dt-\int_{\frac{\varepsilon}{1-x(x-\varepsilon)}}^1\frac{(1-t^2)^{s-1}}{t^{2s}(1-tx)}dt+O(\varepsilon^{2(1-s)})\right\}\\
=&:J_1(x)+J_2(x).
\end{aligned}
$$
As for $J_1(x)$ one can integrate using power series as already done in the case $s<1/2$ and obtain
$$
\begin{aligned}
J_1(x)&=\frac{2c_s}{(1+x^2)^s}\int_0^1\frac{(1-t^2)^{s-1}}{t^{2s-1}}\sum_{k=0}^\infty(tx)^{2k}dt
=\frac{2c_s}{(1+x^2)^s}\sum_{k=0}^\infty \left(x^{2k}\int_0^1\frac{(1-t^2)^{s-1}}{t^{2s-1-2k}}dt\right)\\
&=c_s\frac{\Gamma(s)\Gamma(1-s)}{1-x^2}.
\end{aligned}
$$
We now consider $J_2(x)$. We use again power series to get for every $\varepsilon>0$
$$
\int_{\frac{\varepsilon}{1-x(x+\varepsilon)}}^1\frac{(1-t^2)^{s-1}}{t^{2s}}\frac{1}{1+tx}dt=\sum_{k=0}^\infty(-1)^k x^k\int_{\frac{\varepsilon}{1-x(x+\varepsilon)}}^1\frac{(1-t^2)^{s-1}}{t^{2s-k}}dt
$$
and similarly
$$
\int_{\frac{\varepsilon}{1-x(x-\varepsilon)}}^1\frac{(1-t^2)^{s-1}}{t^{2s}}\frac{1}{1-tx}dt=\sum_{k=0}^\infty x^k\int_{\frac{\varepsilon}{1-x(x-\varepsilon)}}^1\frac{(1-t^2)^{s-1}}{t^{2s-k}}dt. 
$$

We observe that the integrals in the series are all convergent as $\varepsilon\to 0^+$ except for the first ones, where $k=0$. So, we isolate these first terms and calculate, for every $\varepsilon>0$: 
$$
\begin{aligned}
\int_{\frac{\varepsilon}{1-x(x+\varepsilon)}}^1\frac{(1-t^2)^{s-1}}{t^{2s}(1+tx)}dt&-\int_{\frac{\varepsilon}{1-x(x-\varepsilon)}}^1\frac{(1-t^2)^{s-1}}{t^{2s}(1-tx)}dt\\
&=\int_{\frac{\varepsilon}{1-x(x+\varepsilon)}}^1\frac{(1-t^2)^{s-1}}{t^{2s}}dt-\int_{\frac{\varepsilon}{1-x(x-\varepsilon)}}^1\frac{(1-t^2)^{s-1}}{t^{2s}}dt\\
&\quad+\sum_{k=1}^\infty\left\{(-1)^k \int_{\frac{\varepsilon}{1-x(x+\varepsilon)}}^1\frac{(1-t^2)^{s-1}}{t^{2s-k}}dt-\int_{\frac{\varepsilon}{1-x(x-\varepsilon)}}^1\frac{(1-t^2)^{s-1}}{t^{2s-k}}dt\right\}x^k
\end{aligned}
$$
Now, let $F(t)$ be a primitive of $f(t):=\frac{(1-t^2)^{s-1}}{t^{2s}}$, then clearly
\begin{equation}\label{eq:f-f}
\int_{\frac{\varepsilon}{1-x(x+\varepsilon)}}^1 f(t)dt-\int_{\frac{\varepsilon}{1-x(x-\varepsilon)}}^1 f(t)dt = F\left(\frac{\varepsilon}{1-x(x-\varepsilon)}\right)-F\left(\frac{\varepsilon}{1-x(x+\varepsilon)}\right).
\end{equation}
Such a primitive can be expressed in terms of the hypergeometric function $_2F_1$ as follows:
$$
F(t)=\frac{t^{1-2s} {_2F_1}(\frac{1}{2}-s,1-s,\frac{3}{2}-s;t^2)}{1-2s}=\frac{t^{1-2s}}{1-2s}\left(1+O(t^2)\right)\quad\mbox{as }t\to 0.
$$
Inserting this expansion in \eqref{eq:f-f}, by straightforward calculations we get
$$
\int_{\frac{\varepsilon}{1-x(x+\varepsilon)}}^1 f(t)dt-\int_{\frac{\varepsilon}{1-x(x-\varepsilon)}}^1 f(t)dt = -2x\left(\frac{\varepsilon}{1-x^2}\right)^{2-2s}+o(\varepsilon^{2-2s})=O(\varepsilon^{2(1-s)}).
$$
In particular,  $\lim_{\varepsilon\to 0^+}\big(\int_{\frac{\varepsilon}{1-x(x+\varepsilon)}}^1 f(t)dt-\int_{\frac{\varepsilon}{1-x(x-\varepsilon)}}^1 f(t)dt\big)$ is finite. Moreover, we show below that it is finite also the sum of the following series 
$$
\begin{aligned}
\sum_{k=1}^\infty\lim_{\varepsilon\to 0^+}&\left\{(-1)^k \int_{\frac{\varepsilon}{1-x(x+\varepsilon)}}^1\frac{(1-t^2)^{s-1}}{t^{2s-k}}dt-\int_{\frac{\varepsilon}{1-x(x-\varepsilon)}}^1\frac{(1-t^2)^{s-1}}{t^{2s-k}}dt\right\}x^k\\
&=\sum_{k=1}^\infty((-1)^k-1)x^k\int_0^1\frac{(1-t^2)^{s-1}}{t^{2s-k}}dt=\frac{\Gamma(s)}{2}\sum_{k=0}^\infty((-1)^{k+1}-1)x^{k+1}\frac{\Gamma\left(\frac{k-2s+2}{2}\right)}{\Gamma\left(\frac{k+2}{2}\right)}\\
&=\frac{\Gamma(s)}{2}x\sum_{k=0}^\infty(-2)x^{2k}\frac{\Gamma\left(\frac{2k-2s+2}{2}\right)}{\Gamma\left(\frac{2k+2}{2}\right)}=-\Gamma(s)x\sum_{k=0}^\infty x^{2k}\frac{\Gamma(k-s+1)}{\Gamma(k+1)}\\
&=-\Gamma(s)\Gamma(1-s)x\sum_{k=0}^\infty {{s-1}\choose{k}}(-1)^k x^{2k},
\end{aligned}
$$
where we have calculated the integral $\int_0^1\frac{(1-t^2)^{s-1}}{t^{2s-k}}dt$ symbolically by Wolfram Mathematica \cite{Mathematica}, and we have used the sum of the series $\sum_{k=0}^\infty x^{2k}\frac{\Gamma(k-s+1)}{\Gamma(k+1)}$ already calculated for the case $s<1/2$.
Therefore, it is possible to pass to the limit as $\varepsilon\to 0^+$ in the expression of $J_2(x)$ under the series, to get altogether,
$$
\begin{aligned}
J_2(x)&=\frac{c_sx}{(1-x^2)^s}\lim_{\varepsilon\to 0^+}\left\{\int_{\frac{\varepsilon}{1-x(x+\varepsilon)}}^1\frac{(1-t^2)^{s-1}}{t^{2s}(1+tx)}dt-\int_{\frac{\varepsilon}{1-x(x-\varepsilon)}}^1\frac{(1-t^2)^{s-1}}{t^{2s}(1-tx)}dt+O(\varepsilon^{2(1-s)})\right\}\\
&=\frac{c_sx}{(1-x^2)^s}\left\{\lim_{\varepsilon\to 0^+}\left(O(\varepsilon^{2(1-s)})\right)+\sum_{k=1}^\infty((-1)^k-1)x^k\int_0^1\frac{(1-t^2)^{s-1}}{t^{2s-k}}dt\right\}\\
&=-c_s \Gamma(s)\Gamma(1-s)\frac{x^2}{(1-x^2)^s}\sum_{k=0}^\infty {{s-1}\choose{k}}(-1)^kx^{2k}\\
&=-c_s \Gamma(s)\Gamma(1-s)\frac{x^2}{(1-x^2)^s}(1-x^2)^{s-1}
=-c_s \Gamma(s)\Gamma(1-s)\frac{x^2}{(1-x^2)}.
\end{aligned}
$$
In conclusion, 
$$
(-\Delta)^su(x)=J_1(x)+J_2(x)=c_s\frac{\Gamma(s)\Gamma(1-s)}{1-x^2}-c_s \Gamma(s)\Gamma(1-s)\frac{x^2}{(1-x^2)}=c_s\Gamma(s)\Gamma(1-s),
$$
which proves the thesis also in this case.
\end{proof}

\section{The fractional $p$-Laplacian of $(1-|x|^{\frac{p}{p-1}})^s_+$. Preliminaries.}\label{Chap4}
Let $s\in(0,1)$, $p > 1$, and denote by 
$$
u_{s,p}(x):=(1-|x|^{m})^s_+,\quad m:=\frac{p}{p-1}.
$$ 
Having in mind that, for $p=2$, the fractional Laplacian of $u_{s,2}(x)=(1-|x|^2)^s_+$ is constant in $(-1,1)$, see for instance Section \ref{Chap3}, and that $-\Delta_p(1-|x|^{\frac{p}{p-1}})$ is constant in $(-1,1)$, see for instance \cite{CF}, it is tempting to conjecture that also $(-\Delta_p)^s u_{s,p}$ is constant in $(-1,1)$. In the next section we verify numerically that this conjecture is false. 

To this aim, we first prove in this section some preliminary results.

\begin{proposition}\label{prop:Wsp}
For every $s\in (0,1)$ and $p>1$, $u_{s,p}\in W^{s,p}(\mathbb R)$. 
\end{proposition}
\begin{proof}
Clearly, $u_{s,p}(x)=(1-|x|^{\frac{p}{p-1}})^s_+\in L^p(\mathbb R)$. To prove that $u_{s,p}\in W^{s,p}(\mathbb R)$, we need to show that $I:=\int_\mathbb R\int_\mathbb R\frac{|u_{s,p}(x)-u_{s,p}(y)|^p}{|x-y|^{1+sp}}dx\,dy<\infty$. We write the integral under consideration as follows:
$$
\begin{aligned}
I&=2\int_{\mathbb R\setminus (-1,1)}\left(\int_{-1}^1\frac{(1-|y|^{m})^{sp}}{|x-y|^{1+sp}}dy\right)dx + \int_{-1}^1\left(\int_{-1}^1\frac{\left|(1-|x|^{m})^{s}-(1-|y|^{m})^{s}\right|^p}{|x-y|^{1+sp}}dy\right)dx\\ 
&=:2I_1+I_2.
\end{aligned}
$$ 
The integral $I_1$ is convergent. Indeed, arguing as for the integral $I_3(x)$ in the proof of Theorem \ref{thm:frac-tor}, we get 
$$
I_1=\frac{1}{sp}\int_{-1}^1\left(\frac{1}{(1+y)^{sp}}+\frac{1}{(1-y)^{sp}}\right)(1-|y|^{m})^{sp}dy.
$$ 
Moreover, $\frac{(1-y^m)^{sp}}{(1-y)^{sp}}\sim m^{sp}$ as $y\to 1$, and similarly, $\frac{(1-|y|^{m})^{sp}}{(1+y)^{sp}}$ is bounded in a neighborhood of $y=-1$. 
To study the convergence of the integral $I_2$, it is more convenient to change variable and put $t=\frac{x-y}{1-xy}$ in the inner integral, to get
$$
I_2=\int_{-1}^1\left(\int_{-1}^1\left|(1-|x|^m)^s-\left(1-\left|\frac{x-t}{1-tx}\right|^m\right)^s\right|^p\frac{1}{|t|^{1+sp}(1-tx)^{1-sp}}dt\right)\frac{1}{(1-x^2)^{sp}}dx.
$$ 
Now, as $t\to 1$, 
the integrand of the inner integral has the following asymptotics
$$
\frac{\left|(1-|x|^m)^s-\left(1-\left|\frac{x-t}{1-tx}\right|^m\right)^s\right|^p}{|t|^{1+sp}(1-tx)^{1-sp}}\sim \frac{(1-|x|^m)^{sp}}{(1-x)^{1-sp}}
$$
and so, for $t\in (1-\varepsilon,1)$, $I_2$ has the same behavior of  
$$
\int_{-1}^1\frac{(1-|x|^m)^{sp}}{(1-x)^{1-sp}(1-x^2)^{sp}}dx, 
$$ 
which, in view of the fact that 
\begin{equation}\label{eq:asympt}
1-|x|^m=
\begin{cases}
m(x+1)+o(x+1)\quad&\mbox{as }x\to -1,\\
m(1-x)+o(x-1)\quad&\mbox{as }x\to 1,
\end{cases}
\end{equation}
is convergent. 

On the other side, as $t\to0$,
\begin{equation}\label{eq:asympt0}
\begin{gathered}
\frac{x-t}{1-tx}=(x-t)(1+tx+o(t))=x-t(1-x^2)+o(t)\\
\left|\frac{x-t}{1-tx}\right|^m=|x-t(1-x^2)+o(t)|^m=|x|^m(1-m\frac{1-x^2}{x}t+o(t))\\
\left(1-\left|\frac{x-t}{1-tx}\right|^m\right)^s=(1-|x|^m)^s\left(1+ms\frac{|x|^m(1-x^2)}{x(1-|x|^m)}t+o(t)\right)\\
\left|(1-|x|^m)^s-\left(1-\left|\frac{x-t}{1-tx}\right|^m\right)^s\right|^\alpha\sim (1-|x|^m)^{s\alpha}(ms)^\alpha\frac{|x|^{(m-1)\alpha}(1-x^2)^\alpha}{(1-|x|^m)^\alpha}|t|^\alpha,
\end{gathered}
\end{equation}
for any $\alpha>0$. 
Therefore, the integrand of the inner integral (in $dt$) of $I_2$ has the following asymptotics as $t\to 0$
$$
\frac{\left|(1-|x|^m)^s-\left(1-\left|\frac{x-t}{1-tx}\right|^m\right)^s\right|^p}{|t|^{1+sp}(1-tx)^{1-sp}}\sim (1-|x|^m)^{sp}(ms)^p\frac{|x|^{(m-1)p}(1-x^2)^p}{(1-|x|^m)^p}\frac{1}{|t|^{1-p(1-s)}},
$$
and so the integral in $dt$ is convergent. Finally, for $t$ in a neighborhood of $0$, $I_2$ has the same behavior of 
$$
\int_{-1}^1 \frac{|x|^{(m-1)p}(1-x^2)^{p(1-s)}}{(1-|x|^m)^{p(1-s)}}dx,
$$
which again in view of \eqref{eq:asympt} converges.
\end{proof}

\begin{remark}\label{Remarkone}
Arguing as in the first part of the proof of Theorem \ref{thm:frac-tor}, for the fractional $p$-Laplacian of $u_{s,p}=(1-|x|^m)^s_+$, it is possible to calculate explicitly its value at $x=0$. Indeed, denoting by $c_{s,p}$ the normalization constant involved in the definition of the fractional $p$-Laplacian in dimension $1$, we get for every $x\in (-1,1)$
$$
\begin{aligned}
\frac{(-\Delta_p)^su_{s,p}(x)}{c_{s,p}}&=|u_{s,p}(x)|^{p-2}u_{s,p}(x)\int_{\mathbb{R}\setminus(-1,1)}\frac{1}{|x-y|^{1+sp}}dy\\
&\quad
+\lim_{\varepsilon\to 0^+}\int_{(-1,1)\setminus(-\varepsilon,\varepsilon)}\frac{|u_{s,p}(x)-u_{s,p}(y)|^{p-2}(u_{s,p}(x)-u_{s,p}(y))}{|x-y|^{1+sp}}dy\\
&=u_{s,p}(x)^{p-1}\left(\int_{-\infty}^{-1}\frac{1}{|x-y|^{1+sp}}dy+\int_1^{\infty}\frac{1}{|x-y|^{1+sp}}dy\right)\\
&\quad+\lim_{\varepsilon\to 0^+}\int_{(-1,1)\setminus(-\varepsilon,\varepsilon)}\frac{|u_{s,p}(x)-u_{s,p}(y)|^{p-2}(u_{s,p}(x)-u_{s,p}(y))}{|x-y|^{1+sp}}dy\\
&=\frac{(1-|x|^m)^{s(p-1)}}{sp}\left(\frac{1}{(1+x)^{sp}}+\frac{1}{(1-x)^{sp}}\right)\\
&\quad+\lim_{\varepsilon\to 0^+}\int_{(-1,1)\setminus(-\varepsilon,\varepsilon)}\frac{|u_{s,p}(x)-u_{s,p}(y)|^{p-2}(u_{s,p}(x)-u_{s,p}(y))}{|x-y|^{1+sp}}dy.
\end{aligned}
$$
At $x=0$, the previous expression becomes
\begin{equation}\label{eq:x=0}
\begin{aligned}
\frac{(-\Delta_p)^su_{s,p}(0)}{c_{s,p}}&=\frac{2}{sp}+\lim_{\varepsilon\to 0^+}\int_{(-1,1)\setminus(-\varepsilon,\varepsilon)}\frac{(1-(1-|y|^m)^s)^{p-1}}{|y|^{1+sp}}dy\\
&=\frac{2}{sp}+2\int_0^1\frac{(1-(1-y^m)^s)^{p-1}}{y^{1+sp}}dy,
\end{aligned}
\end{equation}
where the integral on the last line is meant in the generalized sense, it is convergent and, at least for some values of $s$ and $p$, can be explicitly expressed in terms of special functions. The value in \eqref{eq:x=0} can be taken as reference value for the numerical analysis.
\end{remark}

\begin{proposition}\label{prop:gxy}
For every $p>1$, $s\in(0,1-1/p)$, and for every $x\in (-1,1)$, the function 
\begin{equation}\label{eq:gx}
g_x(y):=\frac{|u_{s,p}(x)-u_{s,p}(y)|^{p-2}(u_{s,p}(x)-u_{s,p}(y))}{|x-y|^{1+sp}}
\end{equation}
has finite integral over $(-1,1)$. Moreover, if $2-p(1-s)<0$, $g_x$  belongs to the space $W^{r,q}(\mathbb R)$ whenever $q\ge 1$ and $0\le r<\min\{1,p(1-s)-2\}$.
\end{proposition}
\begin{proof}
Fix any $x\in(-1,1)$. Via the usual change of variable $t=\frac{x-y}{1-xy}$, we get
$$
\begin{aligned}
\int_{-1}^1 g_x(y)dy&=\int_{-1}^1\frac{\left|(1-|x|^m)^s-\left(1-\left|\frac{x-t}{1-tx}\right|^m\right)^s\right|^{p-2}\left[(1-|x|^m)^s-\left(1-\left|\frac{x-t}{1-tx}\right|^m\right)^s\right]}{|t|^{1+sp}\left(\frac{1-x^2}{1-tx}\right)^{sp}}dt\\
&=:\int_{-1}^1 f_x(t)dt.
\end{aligned}
$$
Using \eqref{eq:asympt0}, with $\alpha=p-2$, we have that $f_x(t)\sim c(x)\frac{|t|^{p-2}t}{|t|^{1+sp}}$ as $t\to 0$, and so the integral is finite. 
Now, in order to prove the last part of the statement, we write 
$$
\int_\mathbb{R}|g_x(y)|^q dy = \int_{\mathbb{R}\setminus(-1,1)}\frac{(1-|x|^m)^{s(p-1)q}}{|x-y|^{(1+sp)q}}dy+\int_{-1}^1|g_x(y)|^qdy.
$$
The first integral in the sum is finite, being $x\in (-1,1)$ fixed, and $y\not\in(-1,1)$. Concerning the second one, we re-write it arguing as in the first part of this proof
$$
\int_{-1}^1|g_x(y)|^q dy=\int_{-1}^1 |f_x(t)|^q dt, 
$$
and use that $|f_x(t)|^q\sim \frac{c(x)^q}{|t|^{(1+sp-(p-1))q}}$ as $t\to 0$, to conclude that $g_x(y)\in L^q(\mathbb R)$ whenever $(2-p(1-s))q<1$. In particular, $g_x(y)\in L^q(\mathbb R)$ for every $q\ge 1$, if $2-p(1-s)<0$. 
We need to show now that the following integral is finite
$$
\begin{aligned}
\int_\mathbb{R}\int_\mathbb{R}\frac{|g_x(y)-g_x(z)|^q}{|y-z|^{1+rq}}dy\,dz&=\int_{\mathbb{R}\setminus(-1,1)}\int_{\mathbb{R}\setminus(-1,1)}\dots \, dy\,dz+2\int_{\mathbb{R}\setminus(-1,1)}\int_{-1}^1\dots\, dy\,dz\\
&\phantom{=} +\int_{-1}^1\int_{-1}^1\dots \, dy\,dz=: I_1+2I_2+I_3
\end{aligned}
$$
for some $r$. To this aim, we observe that the most singular case is when $y,\,z\in (-1,1)$, and both $y\to x$ and $z\to x$. Therefore, we restrict the study only to the last integral in the sum above:
\begin{equation}\label{eq:I_3}
I_3=\int_{-1}^x\left(\int_{-1}^1\frac{|g_x(y)-g_x(z)|^q}{|y-z|^{1+rq}} dy\right)dz+\int_x^1\left(\int_{-1}^1\frac{|g_x(y)-g_x(z)|^q}{|y-z|^{1+rq}} dy\right)dz.
\end{equation}
We consider the first inner integral in $dy$. For every $z\in(-1,x)$ fixed, and $y\to x$:
\begin{equation}\label{eq:asym-gx}
g_x(y)\sim\mathrm{sgn}(x)\left(\frac{ms |x|^{m-1}}{(1-|x|^m)^{1-s}}\right)^{p-1}\frac{|y-x|^{p-2}(y-x)}{|y-x|^{1+sp}}=:c(x)\mathrm{sgn}(y-x)|y-x|^{p(1-s)-2},
\end{equation}
and so, being $2-p(1-s)<0$, 
$$
\frac{|g_x(y)-g_x(z)|^q}{|y-z|^{1+rq}}\sim \frac{|g_x(z)|^q}{|x-z|^{1+rq}}\quad\mbox{as } y\to x.
$$
Thus, integrating now in $dz$, we have that the integral $\int_{-1}^x\left(\int_{-1}^1\frac{|g_x(y)-g_x(z)|^q}{|y-z|^{1+rq}} dy\right)dz$ has the same behavior of 
$$
\int_{-1}^x \frac{|g_x(z)|^q}{|x-z|^{1+rq}}dz.
$$
Now, like in \eqref{eq:asym-gx}, as $z\to x$,
$$
\frac{|g_x(z)|^q}{|x-z|^{1+rq}}\sim \frac{|c(x)\mathrm{sgn}(z-x)|z-x|^{p(1-s)-2}|^q}{|x-z|^{1+rq}}=\frac{|c(x)|^q}{|x-z|^{1+rq+(2-p(1-s))q}}.
$$
Hence, the first double integral in \eqref{eq:I_3} is convergent, being $1+rq+(2-p(1-s))q<1$ by assumption. The proof of the convergence of the second integral is similar and we omit it. 
\end{proof}

\section{Numerical investigation}\label{Chap5}

In this Section we show that $\exists\, p>2$ and $\exists\, s \in (0,1)$ such that (\ref{eq:classical-Deltaps0}) is not constant in $(-1,1)$. To this aim it is sufficient to show that 
\begin{equation}\label{eq:global-integral}
I^{(s,p)}(x)=\lim_{\varepsilon \to 0}\int_{(B_\varepsilon(x))^c} g_x(y)dy
\end{equation}
is not constant, where $g_x$ is the function defined in (\ref{eq:gx}).
We will limit ourselves to provide numerical evidence to this statement.

For sake of clearness, we omit now the indices $s$ and $p$ in $I^{(s,p)}(x)$, noticing that the approximations we are going to present are valid for any $s$ and $p$ for which $I(x)=I^{(s,p)}(x)$ is finite. 
Then we split $I(x)=I^{(s,p)}(x)$  into the sum of six contributions as follows:
\begin{equation}\label{eq:decomposed-itegral}
I(x)=\underbrace{\int_{-\infty}^{-1} \!\!\!\!\! g_x(y)dy}_{I_1(x)}+
\underbrace{\int_{-1}^{x-\varepsilon} \!\!\!\!\!g_x(y)dy}_{I_2(x)}+
\underbrace{\int_{x-\varepsilon}^x \!\!\!\!\! g_x(y)dy}_{I_3(x)}+
\underbrace{\int_x^{x+\varepsilon}\!\!\!\!\!  g_x(y)dy}_{I_4(x)}+
\underbrace{\int_{x+\varepsilon}^1 \!\!\!\!\! g_x(y)dy}_{I_5(x)}+
\underbrace{\int_1^{+\infty} \!\!\!\!\! g_x(y)dy}_{I_6(x)},
\end{equation}
where $\varepsilon>0$ will be specified later.

The most challenging integrals to compute are $I_3(x)$ and $I_4(x)$ because of the presence of the singularity of $g_x(y)$ at $y=x$.

From now on, we denote by $\tilde I_k(x)$ the numerical approximation of the integral $I_k(x)$ for $k=1,\ldots,6$.

The integrals $I_1(x)$ and $I_6(x)$ are approximated by an adaptive quadrature formula implemented in the functions {\tt integral} and {\tt quadva} of MATLAB \cite{Shampine_2008}, after operating a  change of variable that transforms them to integrals on a finite interval with a very mild singularity.
The approximated integrals $\tilde{I}_1(x)$ and $\tilde I_6(x)$  are computed by ensuring that
\begin{equation}\label{eq:stima-errore-integrali16}
|I_k(x)-\tilde I_k(x)|\leq 10^{-15} \qquad \qquad \mbox{ for }k\in\{1,6\}.
\end{equation}

Since we are performing our computations with double-precision arithmetic for which the machine precision is about $10^{-16}$, the tolerance of $10^{-15}$ in (\ref{eq:stima-errore-integrali16}) is fully satisfactory.

\medskip
The approximate integrals $\tilde I_2(x),\ldots,\tilde I_5(x)$ are computed by the Gauss--Legendre quadrature formula using $(n+1)$ nodes (see, e.g. \cite[(2.3.10)]{chqz07}).
To highlight the dependence of the computed integrals on the number of nodes, we use the notation $\tilde I_{k,n}(x)$ instead of $\tilde I_{k}(x)$, for $k\in\{2,3,4,5\}$.

For what concerns the numerical error of the Gauss--Legendre quadrature formula, it is possible to prove that there exists a positive constant $C$ only depending on the size of the integration interval such that, for $k=2,\ldots,5$  and for any $x\in (-1,1)$, it holds
\begin{equation}\label{eq:stima-errore-integrali25}
|I_k(x)-\tilde I_{k,n}(x)|\leq C n^{-\sigma} \|g_x\|_{W^{\sigma,2}(\Lambda_k)},
\end{equation}
provided that $g_x\in W^{\sigma,2}(\Lambda_k)$ for some $\sigma>1/2$ and where
$\Lambda_k$ denotes the integration interval of the integral $I_k(x)$.
The proof of (\ref{eq:stima-errore-integrali25}) follows by applying the estimate (5.3.4a) of \cite{chqz07} and the estimate (3.7) of \cite{canuto_quarteroni_82} with Legendre weight $w(y)\equiv 1$.

Then, thanks to the estimates (\ref{eq:stima-errore-integrali16}) and (\ref{eq:stima-errore-integrali25}), it holds that the global approximated integral 
\begin{equation}\label{eq:global-appx-integral}
\tilde I(x)=\sum_{k=1}^6 \tilde I_k(x)
\end{equation}
satisfies the estimate
\begin{equation}\label{eq:stima-errore-globale}
|I(x)-\tilde I(x)|\leq c n^{-\sigma }
\|g_x\|_{W^{\sigma,2}(\Lambda_k)}+10^{-15},
\end{equation}
i.e., $\tilde I(x)$ converges to the exact value $I(x)$ when $n\to \infty$, for any $x\in(-1,1)$ up to the tolerance $\overline\epsilon=10^{-15}$.

To get it, it is sufficient to take a number $(n+1)$ of quadrature nodes sufficiently large to guarantee that the error $|I(x)-\tilde I(x)|$ be small enough. Since the value of $I(x)$ is unknown when $p\neq 2$, but it is known when $p=2$, we take the case $p=2$ as a playground to learn how many quadrature nodes we need to consider in order to approximate $I(x)$ with the desired accuracy. 

Let us now resume the original notation of $I^{(s,p)}(x)$ because we are interested in distinguishing what happens for different values of $p$ and $s$.

\subsection{The case $p=2$}
When $p=2$ we know that  (see the proof of Theorem \ref{thm:frac-tor})
\begin{equation}\label{eq:Is2}
I^{(s,2)}(x)=\frac{\pi}{\sin(\pi s)}.
\end{equation}

In Figure \ref{fig:integrali-p2}, left,  we plot the values $\tilde
I^{(s,2)}(x)$ , for several values of $x\in(-1,1)$ and for $s\in\{0.2,\ 0.4,\ 0.5,\ 0.58\overline{3}\}$. 
We have chosen $\varepsilon=1/50$ in (\ref{eq:decomposed-itegral}). Numerical results are fully consistent with the theoretical result reported in (\ref{eq:Is2}), the values $\frac{\pi}{\sin(\pi s)}$ are represented by the empty squares (only in correspondence of $x=0$).

In Figure \ref{fig:integrali-p2}, right, we report the absolute errors $|I^{(s,2)}(x)-\tilde I^{(s,2)}(x)|$ for several values of $x\in(-1,1)$.
When $s=0.2$, $s=0.4$, and $s=0.5$, the errors are all below $5\cdot 10^{-6}$.
Instead, when $s=0.58\overline{3}$, the errors are about $10^{-6}$ in the middle of the
interval and reach the value $10^{-4}$ when $|x|$ tends to 1. We explain this
behavior to the fact that when $s\to 1^-$, the order of infinity of the
function $g_x(y)$ at $y=x$ increases and the computation of the corresponding
integral is very demanding.

\begin{figure}
\begin{center}
    \includegraphics[width=0.48\textwidth]{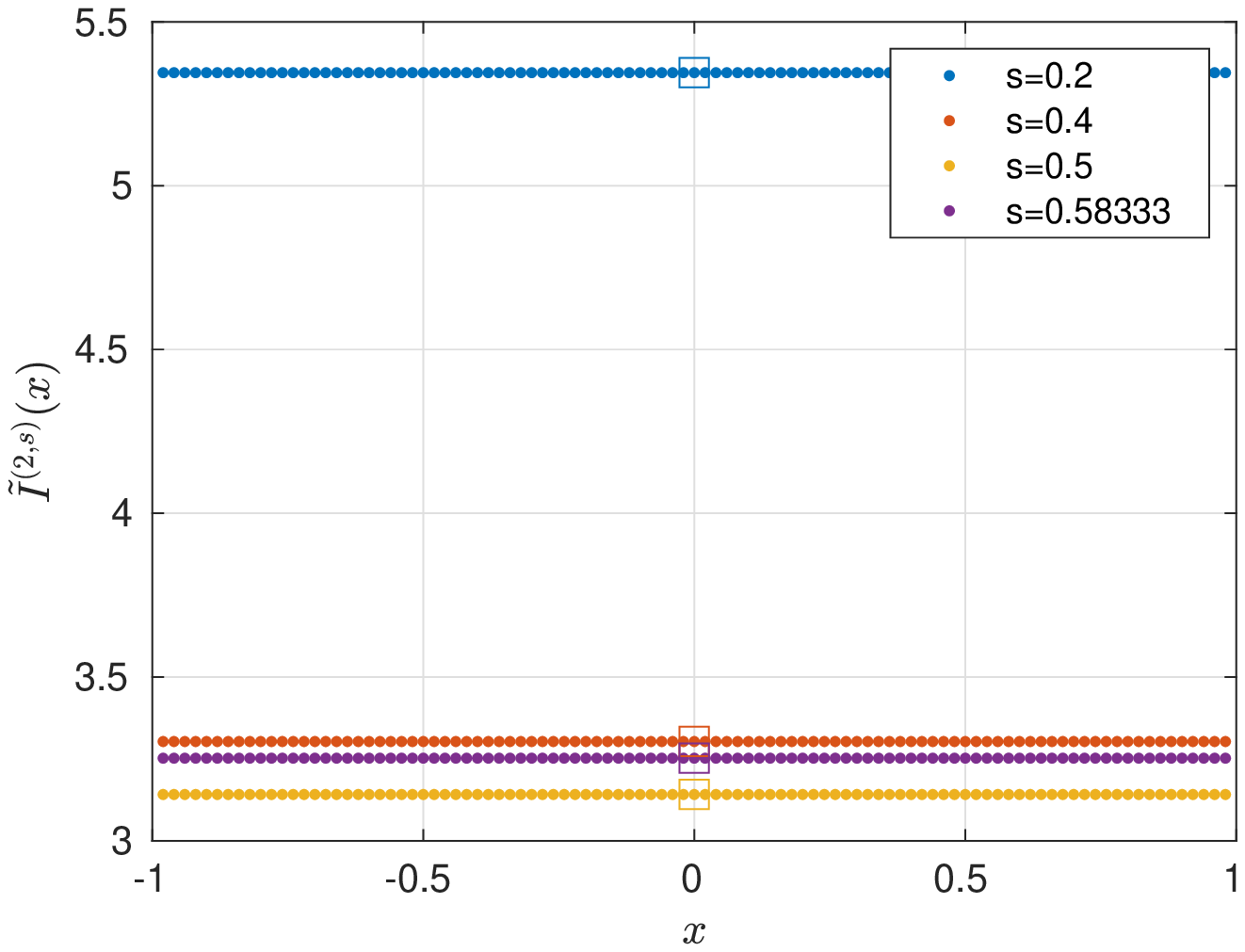}\quad
    \includegraphics[width=0.48\textwidth]{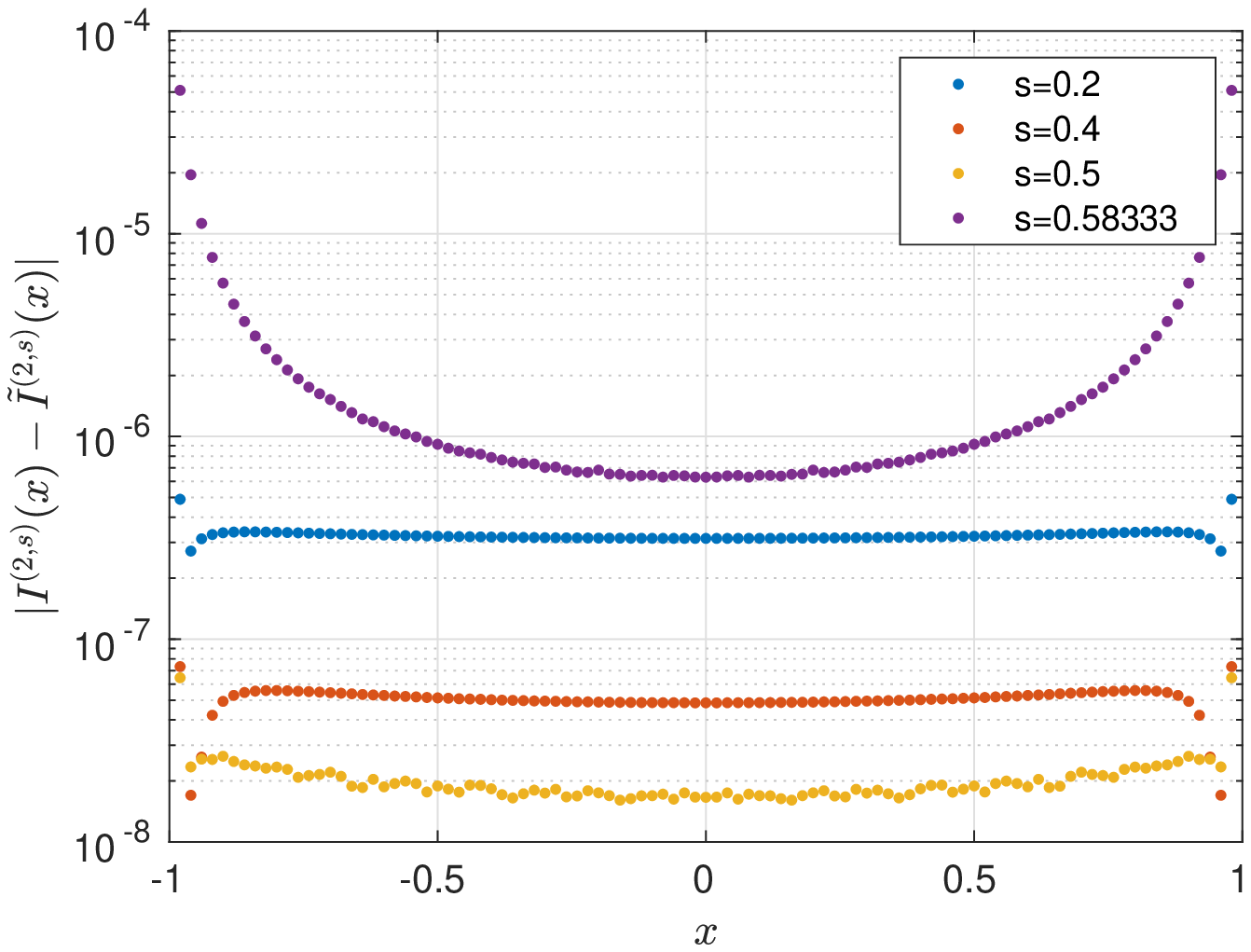}
\end{center}
\caption{On the left, the approximated  integrals $\tilde I^{(s,2)}(x)$, the empty squares at $x=0$ represent the values (\ref{eq:Is2}). On the right, the absolute errors $|I^{(s,2)}(x)-\tilde I^{(s,2)}(x)|$}
\label{fig:integrali-p2}
\end{figure}

In Figure \ref{fig:errori-p2-vs-n} we show the behavior of the errors $|I^{(s,2)}(x)-\tilde I^{(s,2)}(x)|$ versus $n$, and for different values of $s$, at $x=0$ (left) and $x=0.5$ (right).

 When $p=2$ there is no value of $s>1/2$ for which we know that $g_x\in W^{s,2}({\mathbb
R})$ (see Proposition \ref{prop:gxy}), hence we cannot take advantage of the
estimate (\ref{eq:stima-errore-integrali25}). Yet, we observe that the errors
for all the values of $s$ decrease when $n$ grows up, showing convergence of
the approximated integrals to the exact ones. The value $n=256$ provides very
satisfactory results: all the errors are lower than $10^{-6}$. Moreover, we can conclude that the accuracy of the quadrature formula at $x=0$ and $x=0.5$ is almost the same for $n$ ranging between $64$ and $256$.

\begin{figure}
\begin{center}
    \includegraphics[width=0.48\textwidth]{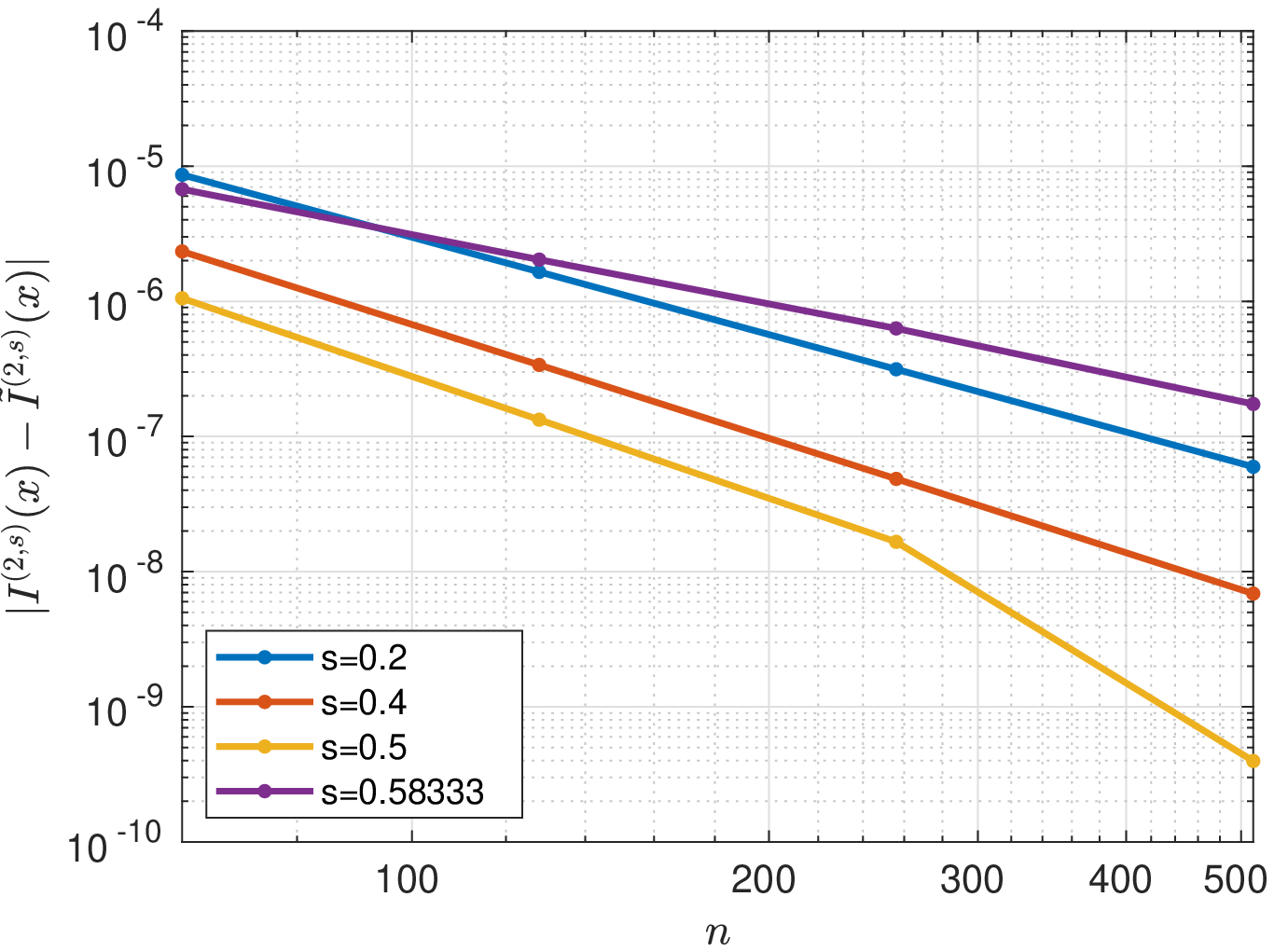}\quad
    \includegraphics[width=0.48\textwidth]{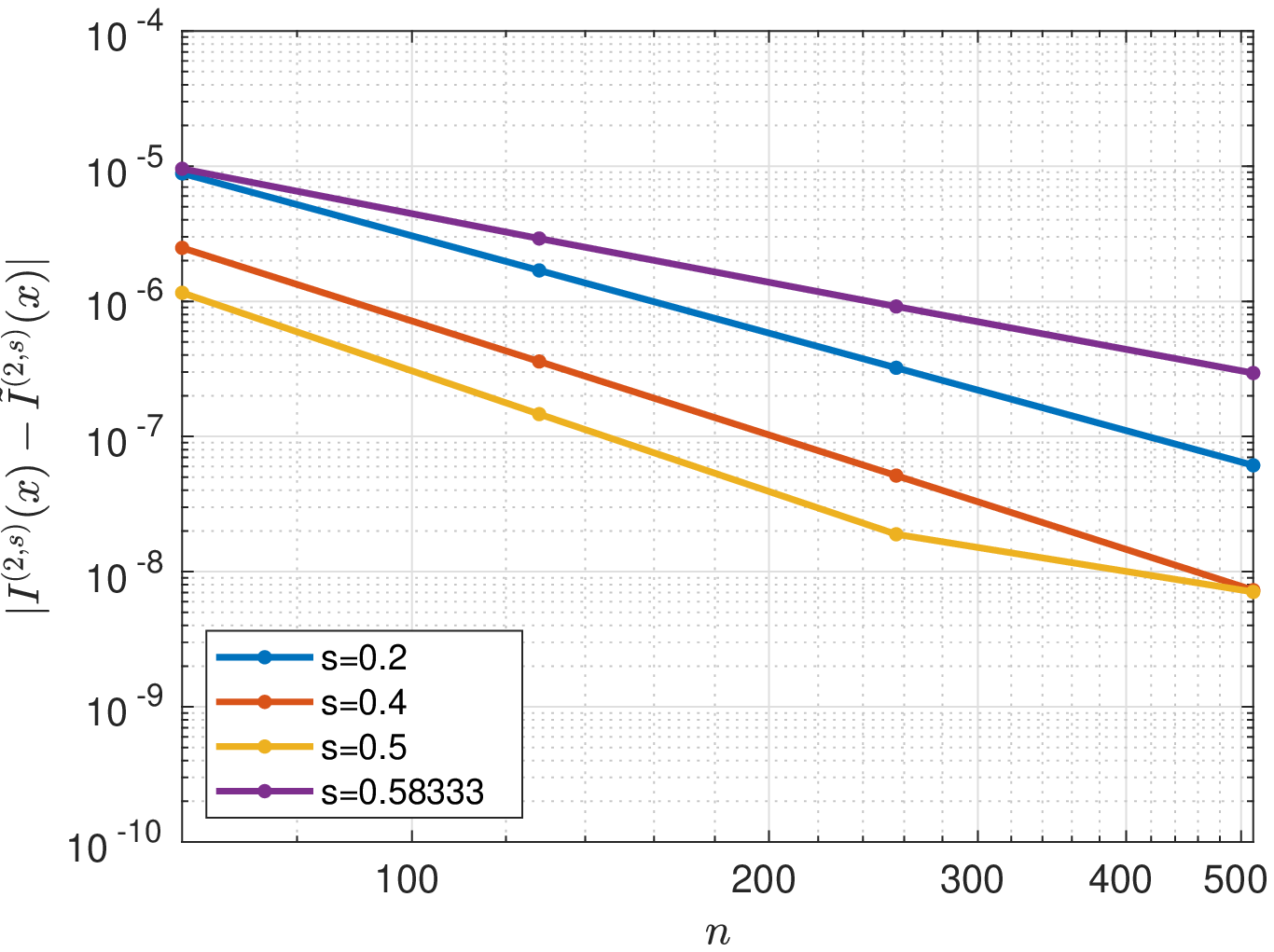}
\end{center}
\caption{The absolute errors $|I^{(s,2)}(x)-\tilde I^{(s,2)}(x)|$ versus $n$ for different values of $s$. On the left at $x=0$, on the right at $x=0.5$}
\label{fig:errori-p2-vs-n}
\end{figure}

\subsection{The case $p\neq 2$}

So far, we have tested the accuracy of our quadrature formulas; now  we can move to the case $p\neq 2$, for which we only know the exact value of the integral $I^{(s,p)}(x)$ when $x=0$. As a matter of fact, we have (see  (\ref{eq:x=0}))
\begin{equation}\label{eq:integrale-esatto-p-0}
I^{(s,p)}(0)=\frac{2}{sp}+2\int_0^1\frac{(1-(1-y^m)^s)^{p-1}}{y^{1+sp}}dy
\end{equation}
and we have computed it symbolically by Wolfram Mathematica \cite{Mathematica}.

In Figure \ref{fig:integrali-p3}, left, we report the values of $I^{(s,p)}(x)$ when $p=3$, for five values of $s$ and different values of $x\in(-1,1)$. Clearly,  $I^{(s,p)}(x)$ is not constant in $(-1,1)$. The square symbols at $x=0$ represent the exact values 
(\ref{eq:integrale-esatto-p-0}). In the right picture of Figure \ref{fig:integrali-p3} we display the errors $|I^{(s,3)}(0)-\tilde{I}^{(s,3)}(0)|$ for five values of $s$ versus the parameter $n$ (related to the number of quadrature nodes). When $n$ increases all the errors decrease with rate comparable with that for the case $p=2$  (see Figure \ref{fig:errori-p2-vs-n}). Then we expect that the same accuracy occur in correspondence to other points $x\neq  0$ that stand sufficiently far from the end-points of the interval $(-1,1)$.  
Differently than for the case $p=2$, here we have reported numerical results also for $s=2/15$, so that $g_x\in W^{r,2}({\mathbb R})$ with $r=0.6$, and for which the estimate (\ref{eq:stima-errore-integrali25}) holds.

\begin{figure}
\begin{center}
    \includegraphics[width=0.48\textwidth]{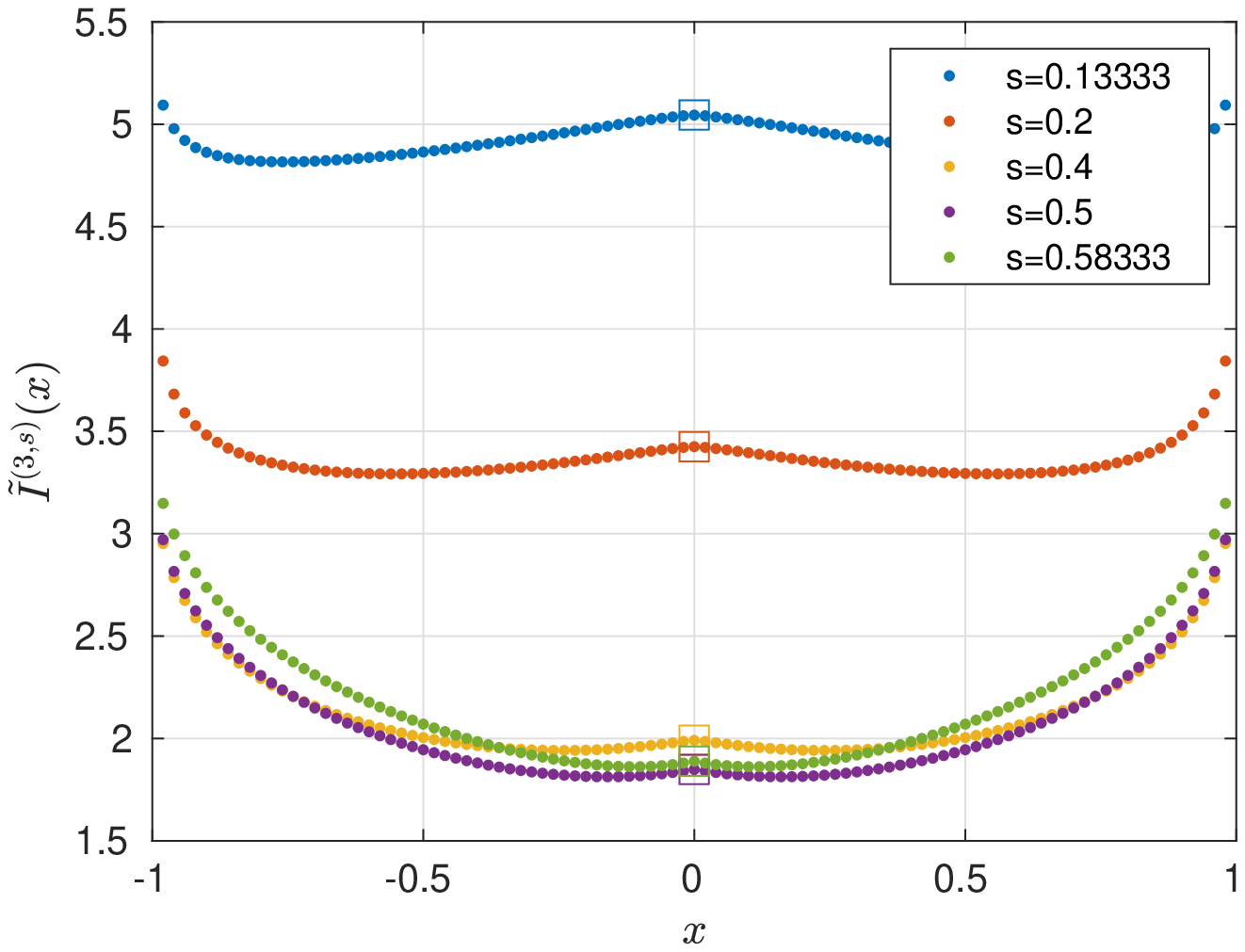}\quad
    \includegraphics[width=0.48\textwidth]{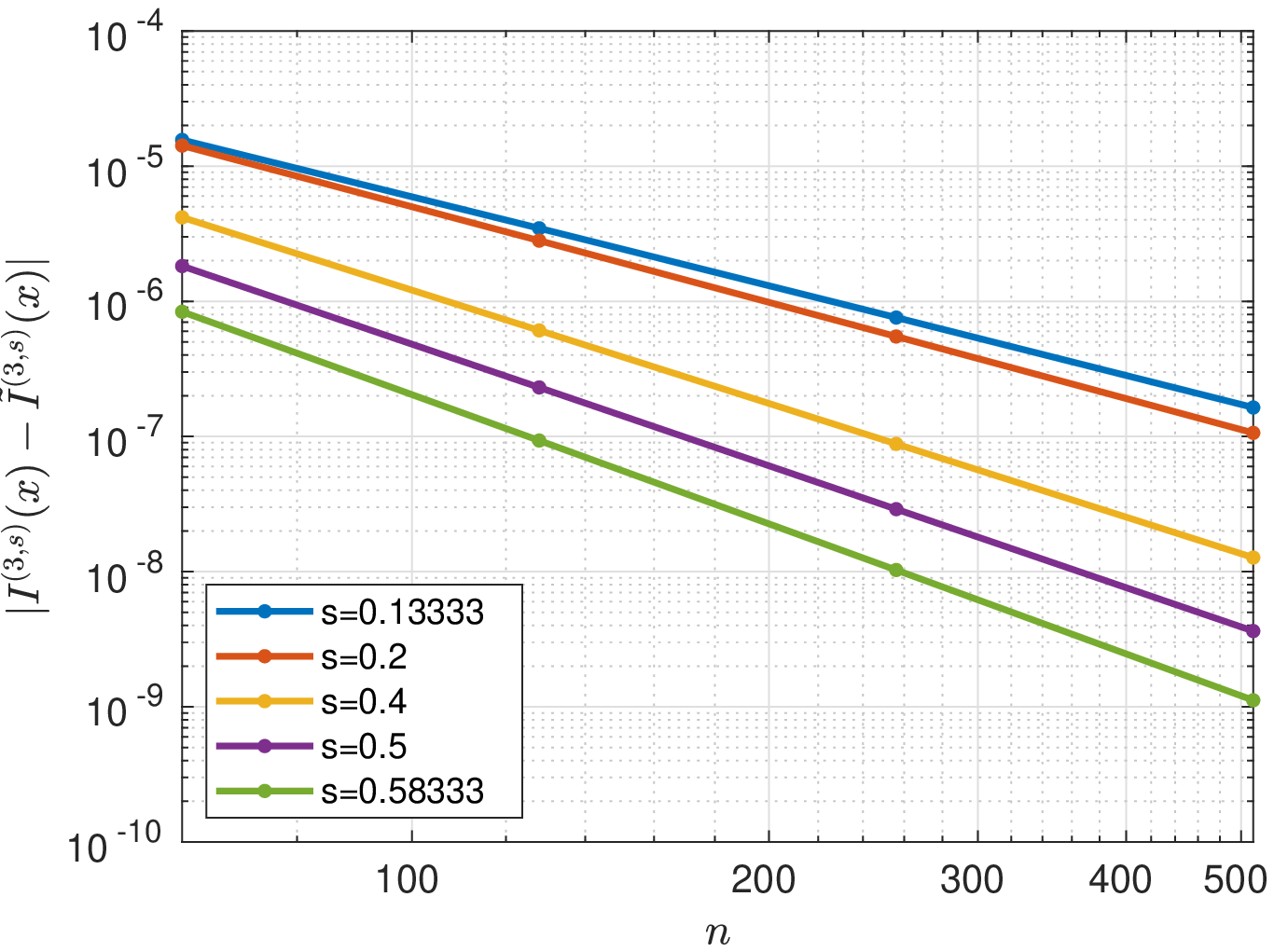}
\end{center}
\caption{On the left, the approximated  integrals $\tilde I^{(s,3)}(x)$. The empty squares at $x=0$ represent the exact values (\ref{eq:integrale-esatto-p-0}). On the right, the absolute errors $|I^{(s,3)}(x)-\tilde I^{(s,3)}(x)|$ at $x=0$. The numerical integrals are evaluated using $(n+1)$ nodes.}
\label{fig:integrali-p3}
\end{figure}

Similar results, but now for $p=4$, are shown in Figure \ref{fig:integrali-p4}: on the left, we report the values of $I^{(s,4)}(x)$ for four values of $s$ and different values of $x\in(-1,1)$. Also in this case it is evident that  $I^{(s,4)}(x)$ is not constant in $(-1,1)$. The square symbols at $x=0$ refer to the exact values 
(\ref{eq:integrale-esatto-p-0}). In the right picture of Figure \ref{fig:integrali-p4} we show the errors $|I^{(s,4)}(0)-\tilde{I}^{(s,4)}(0)|$ for four values of $s$ versus the parameter $n$ (related to the number of quadrature nodes). Similar conclusions made for $p=3$ can be drawn for $p=4$, too. 
 
\begin{figure}
\begin{center}
    \includegraphics[width=0.48\textwidth]{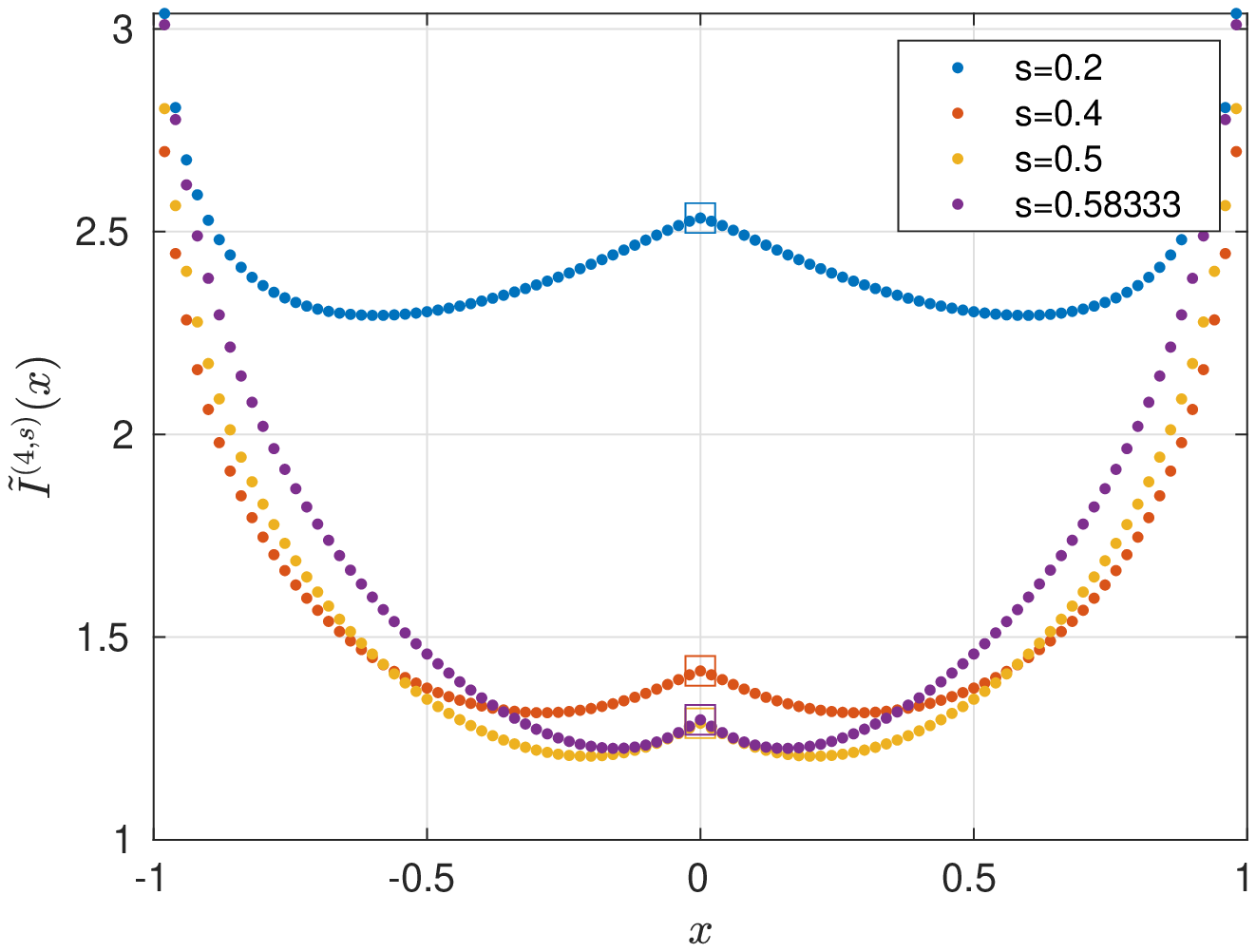}\quad
    \includegraphics[width=0.48\textwidth]{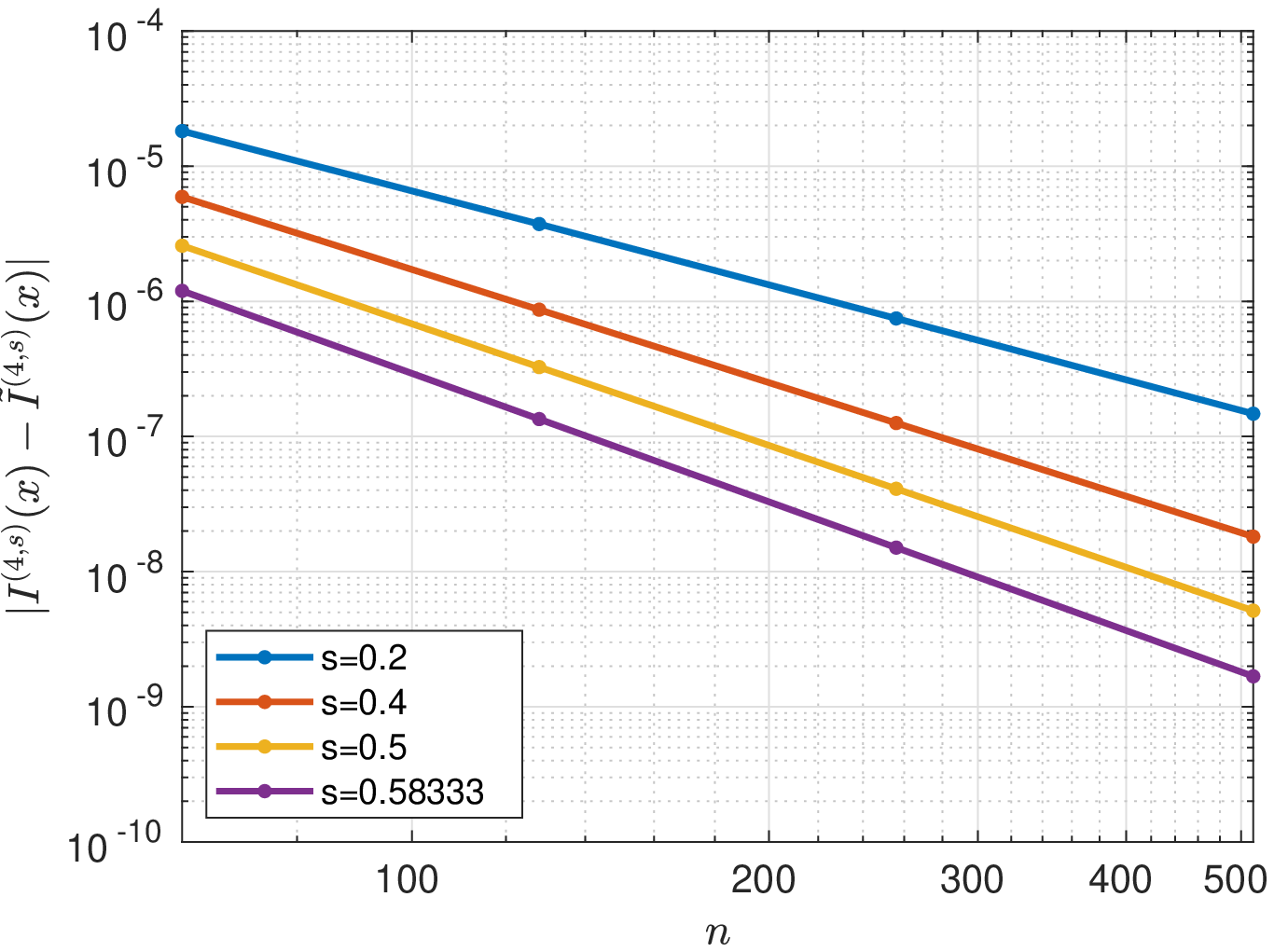}
\end{center}
\caption{On the left, the approximated  integrals $\tilde I^{(s,4)}(x)$. The empty squares at $x=0$ represent the exact values (\ref{eq:integrale-esatto-p-0}). On the right, the absolute errors $|I^{(s,4)}(x)-\tilde I^{(s,4)}(x)|$ at $x=0$. The numerical integrals are evaluated using $(n+1)$ nodes.}
\label{fig:integrali-p4}
\end{figure}

Bearing in mind that when $p=2$ the errors at $x=0$ and $x=0.5$ were substantially the same, for a fixed value of $s$, we can conclude that also when $p>2$ the accuracy in approximating the integrals at $x\neq 0$ is comparable to that obtained at $x=0$.
Moreover, we observe that, for a fixed $s$, the regularity of $g_x(y)$ increases with $p$ and this allows us to benefit of the greater convergence order in the estimate (\ref{eq:stima-errore-integrali25}). This implies that, when $p>2$, we can expect that the approximated integrals are at least accurate as those for $p=2$.

In conclusion, in Table \ref{tab:confronto-finale} we report the approximated values $\tilde I^{(s,p)}(x)$ for $p=3$ and $p=4$, for some values of $s$ and at the two points $x=0$ and $x=0.5$. Because these values approximate the corresponding exact values with errors lower than about $10^{-6}$, we can state once more that $\exists\, p\neq 2$ and $\exists \, s\in(0,1)$ such that $I^{(s,p)}$ is not constant in $(-1,1)$.

\begin{table}\label{tab:confronto-finale}
\begin{center}
\renewcommand{\arraystretch}{1.2}
\begin{tabular}{l|l l}
$s$ & $\tilde I^{(s,3)}(0)$ & $\tilde I^{(s,3)}(0.5)$\\
\hline
$0.1\overline{3}$  &  $5.0446$  & $4.8644$ \\ 
$0.20$  &  $3.4253$  & $3.2945$ \\ 
$0.40$  &  $1.9911$  & $2.0046$ \\ 
$0.50$  &  $1.8484$  & $1.9451$ \\ 
$0.58\overline{3}$  &  $1.8891$  & $2.0702$ \\ 
\end{tabular}
\qquad
 \begin{tabular}{l|l l}
$s$ & $\tilde I^{(s,4)}(0)$ & $\tilde I^{(s,4)}(0.5)$\\
\hline
$0.1\overline{3}$  &  $3.7625$  & $3.4608$ \\ 
$0.20$  &  $2.5335$  & $2.3025$ \\ 
$0.40$  &  $1.4166$  & $1.3743$ \\ 
$0.50$  &  $1.2876$  & $1.3469$ \\ 
$0.58\overline{3}$  &  $1.2962$  & $1.4584$ \\ 
\end{tabular}
\end{center}
\caption{The values of $\tilde I^{(s,p)}(0)$ and $\tilde I^{(s,p)}(0.5)$ for some values of $s$, computed with the formula (\ref{eq:global-appx-integral}) and $n=256$. On the left $p=3$, on the right $p=4$. These values approximate the corresponding exact values with errors lower than $5\cdot 10^{-5}$.}
\end{table}

\section*{Acknowledgments}
F.C. was partially supported by the INdAM - GNAMPA Project 2020 ``Problemi ai limiti per l'equazione della curvatura media prescritta''.

\bibliographystyle{abbrv}
\bibliography{biblio}

\end{document}